\documentclass{article}
\usepackage[utf8]{inputenc}

\usepackage{amssymb,amsmath,amsthm}

\usepackage{float}
\usepackage{booktabs}
\usepackage{pdflscape}
\usepackage{amsmath,mathtools,calc}
\usepackage{subcaption}
\usepackage{algorithm}
\usepackage{algpseudocode}
\usepackage{hyperref}
\usepackage{multirow}
\usepackage{tikz}
\usepackage{comment}
\usepackage{ulem}
\usetikzlibrary{arrows}
\usetikzlibrary{arrows,positioning,automata,calc}
\usepackage{xkeyval}
\newtheorem{thm}{Theorem}
\newtheorem{prop}{Proposition}
\DeclareMathOperator*{\argmax}{argmax}

\begin{document}

\title{Precedence-Constrained Arborescences}

\author{Xiaochen Chou \and
Mauro Dell'Amico \and
Jafar Jamal \and
Roberto Montemanni\footnote{Corresponding author: roberto.montemanni@unimore.it}}

\date{University of Modena and Reggio Emilia, Italy}
\maketitle
\begin{abstract}
The minimum-cost arborescence problem is a well-studied problem in the area of graph theory, with known polynomial-time algorithms for solving it. Previous literature introduced new variations on the original problem with different objective function and/or constraints. Recently, the Precedence-Constrained Minimum-Cost Arborescence problem was proposed, in which precedence constraints are enforced on pairs of vertices. These constraints prevent the formation of directed paths that violate precedence relationships along the tree. We show that this problem is {{\sc NP}-hard}, and we introduce a new scalable mixed integer linear programming model for it. With respect to the previous models, the newly proposed model performs substantially better. This work also introduces a new variation on the minimum-cost arborescence problem with precedence constraints. {We show that this new variation is also {{\sc NP}-hard}, and} we propose several mixed integer linear programming models for formulating the problem.
\end{abstract}


\section{Introduction}
\label{sec-introduction}

The \textit{Minimum-Cost Arborescence problem} (MCA) is a well-known problem that consists in finding a directed minimum-cost spanning tree {rooted at some vertex $r$ called the root} in a directed graph. The first polynomial time algorithm for solving the problem was proposed independently by Yoeng-Jin Chu and Tseng-Hong Liu \cite{ref-chu}, and Jack Edmonds \cite{ref-edmonds}. The problem can be formally described as follows. A directed graph $G=(V,A)$ is given where {$V=\{1,\cdots ,n\}$} is the set of vertices, $r \in V$ is the root of the arborescence, and $A \subseteq V \times V$ is the set of arcs with a cost $c_a$ associated with every arc $a \in A$. The goal is to find a minimum-cost directed spanning tree in $G$ rooted at $r$, i.e. a set $T \subseteq A$ of $n - 1$ arcs, such that there is a unique directed path from $r$ to any other vertex $j \in V\setminus\{r\}$ in the subgraph induced by $T$. A different polynomial time algorithm for solving the MCA that operates directly on the cost matrix was discussed by Bock \cite{ref-bock}. 

Since the MCA was first proposed, different variations have been introduced such as the \textit{Resource-Constrained Minimum-Weight Arborescence problem} \cite{ref-Fischetti}, where finite resources are associated with each vertex in the input graph. The objective of the problem is to find an arborescence with minimum total cost where the sum of the costs of outgoing arcs from each vertex is at most equal to the resource of that vertex. The problem is categorized as {{\sc NP}-hard} as it generalizes the Knapsack problem \cite{ref-Fischetti}. The \textit{Capacitated Minimum Spanning Tree problem} \cite{ref-Gouveia} is another variation, where non-negative integer node demands $q_j$ is associated with each node $j \in  V \backslash \{r\}$, and an integer $Q$ is given. The objective is to find a minimum spanning tree rooted at $r$ such that the sum of the weights of the vertices in any subtree off the root is at most $Q$. The problem is shown to be {{\sc NP}-hard} as the particular case with zero cost arcs is a \textit{bin packing problem}  \cite{ref-Gouveia}. The \textit{$p$-Arborescence Star problem} \cite{ref-Pereira} is a relevant problem that is described as follows. Given a weighted directed graph $G =(V, A)$, a root vertex $r \in V$, and an integer $p$, the objective of the problem is to find a minimum-cost reverse arborescence rooted at $r$, such that the arborescence {spans} the set of vertices $H \subseteq V \backslash \{r\}$ of size $p$, and each vertex $v \in V\backslash \{H\cup r\}$ must be assigned to one of the vertices in $H$. The problem is {{\sc NP}-hard} \cite{ref-Morais} in the general case by a reduction from the \textit{$p$-median problem} \cite{ref-Daskin}. Frieze and Tkocz \cite{ref-Frieze} study the problem of finding a minimum-cost arborescence such that the cost of the arborescence is at most $c_0$. The problem is studied on randomly weighted digraphs where each arc in the graph has a weight $w$ and a cost $c$, each being an independent uniform random variable $U^s$ where $0 < s \leq 1$, and $U$ is uniform $[0,1]$. The problem is {{\sc NP}-hard} \cite{ref-Frieze} through a reduction from the knapsack problem. Another problem is the \textit{Maximum Colorful Arborescence problem} \cite{ref-Fertin} which can be described as the following. Given a weighted directed acyclic graph with each vertex having a specified color from a set of colors $C$, the objective is to find an arborescence of maximum weight, in which no color appears more than once. The problem is known to be {{\sc NP}-hard} \cite{ref-Bocker} even when all arcs have a weight of 1. \textit{The Constrained Arborescence Augmentation problem} \cite{ref-Li} is a different variation that can be described as follows. Given a weighted directed graph $G = (V, A)$, and an arborescence $T = (V, A_r)$ in $G$ rooted at vertex $r \in V$, the objective of the problem is to find an arc subset $A'$ from $A - A_r$ such that there still exists an arborescence in the new graph $G' = (V, A_r \cup A' - {a})$ for each arc $a \in A_r$, where the sum of {the weights of the arcs} in $A'$ is minimized. The problem is an extension on the \textit{augmentation problem} \cite{ref-Eswaran}, and is shown to be {\sc NP}-{hard} \cite{ref-Li}. The \textit{Minimum $k$ Arborescence with Bandwidth Constraints} \cite{ref-cai} is another variation, where every arc $a \in A$ has an integer bandwidth $b(a)$ that indicates the number of times such an arc can be used. The objective of the problem is to find $k$ arborescences of minimum-cost rooted at the $k$ given root vertices, covering every arc $a \in A$ at most $b(a)$ times. It has been shown that the problem can be solved in polynomial time \cite{ref-cai}. \textit{The Degree-Constrained Minimal Spanning Tree problem with unreliable links and node outage costs} \cite{ref-Kawatra} is modeled as a directed graph with the root vertex being the central node of a network, and all other vertices {being terminal nodes}. The problem {consists in} finding links in a network to connect a set of terminal nodes to a central node, while minimizing both link costs and node outage costs. Node outage cost is the economic cost incurred by the network user whenever that node is disabled due to failure of a link. The problem is shown to be {{\sc NP}-hard} by reducing the problem to an equivalent \textit{Traveling Salesman problem} \cite{ref-Garey}. The \textit{Minimum Changeover Cost Arborescence} \cite{ref-changeover} is another variation, where each arc is labeled with a color out of a set of $k$ available colors. A changeover cost is defined on every vertex $v$ in the arborescence other than the root. The cost over a vertex $v$ is paid for each outgoing arc from $v$ and depends on the color of its outgoing arcs, relative to the color of its incoming arc. The costs are given through a $k \times k$ matrix $C$, where each entry $C_{ab}$ specifies the cost to be paid at vertex $v$ when its incoming arc is colored $a$ and one of its outgoing arcs is colored $b$. A change over cost at vertex $v$ is calculated as the sum of costs paid for every outgoing arc at that vertex. The objective of the problem is to find an arborescence $T$ with minimum total change over cost for every vertex $j \in V$ other than the root. The problem is shown to be {\sc NP}-{hard} and very hard to approximate \cite{ref-changeover}. Finding a pair of arc-disjoint in-arborescence and out-arborescence is another problem, with the objective of finding a pair of arc-disjoint $r$-out-arborescence rooted at $r_1$ and $r$-in-arborescence rooted at $r_2$ where $r_1,r_2 \in V$. An $r$-out-arborescence has all its arcs directed away from the root, and an $r$-in-arborescence has all its arcs directed towards the root. The problem was studied by Bérczi et al. \cite{ref-arc-disjoint} where a linear-time algorithm for solving the problem in directed acyclic graphs is proposed. The problem is shown to be {\sc NP}-Complete in general graphs even if $r_1 = r_2$ \cite{ref-edge-disjoing-NP}. Yingshu et al. \cite{ref-broadcast} studied the problem of constructing a strongly connected broadcast arborescence with bounded transmission delay, where they devise a polynomial time algorithm for constructing a broadcast network with minimum energy consumption that respects the transmission delays of the broadcast tree simultaneously. {The \textit{Minimum Spanning Tree Problem with Conflict Pairs} is a variation of the minimum spanning tree problem where given an undirected graph and a set $S$ that contains \emph{conflicting pairs} of edges called a \textit{conflict pair}, the objective of the problem is to find a minimum-cost spanning tree that contains at most one edge from each conflict pair in $S$ \cite{ref-mstcp}. The problem is shown to be {\sc NP}-hard \cite{ref-mstcp-np}. The \textit{Least-Dependency Constrained Spanning Tree problem} \cite{ref-cmst} is another variation that can be defined as follows. Given a connected graph $G=(V,E)$ and a directed graph $D=(E,A)$ whose vertices are the edges of $G$, the directed graph $D$ is a \emph{dependency graph} for $E$, and $e_1 \in E$ is a dependency of $e_2 \in E$ if $(e_1,e_2) \in A$. The objective of the problem is to decide whether there is a spanning tree $T$ of $G$ such that each edge in $T$ has either an empty dependency or \emph{at least} one of its dependencies is also in $T$. The \textit{All-Dependency Constrained Spanning Tree problem} \cite{ref-cmst} is a similar problem that consists in deciding whether there is a spanning tree $T$ of $G$ such that each of its edges either has no dependency or {all of its dependencies}  are in $T$. The two problems are shown to be {\sc NP}-Complete \cite{ref-cmst}.}

The \textit{Precedence-Constrained Minimum-Cost Arborescence problem} \linebreak (PCMCA) was first introduced by Dell'Amico et al. \cite{ref-dellamico}, where a set of precedence constraints is included as follows. Given a set $R$ of ordered pairs of vertices, then for each precedence $(s,t)\in R$ any path of the arborescence covering both vertices $s$ and $t$ must visit $s$ before visiting $t$. The objective of the problem is to find an arborescence of minimum total cost that satisfies the precedence constraints. By definition of the PCMCA, we always assume that if $(s, t) \in R$ then $(t, s) \notin A$. The PCMCA has applications in infrastructure design such as designing a commodity distribution network. As an example, assume we have a commodity distribution network, where the distribution starts from a main vertex (root of the arborescence), and the distribution travels in a single direction away from the root to every other vertex in the graph. Such a structure follows the definition of an arborescence. Now assume that transit duties that are higher than the travel costs have to be paid by vertex $s$ in the graph, if the commodity passes through vertex $t$ on its way to vertex $s$. To avoid for such duties to be paid by vertex $s$, we can impose a precedence relationship between the vertex pair $s$ and $t$, i.e. $(s,t) \in R$. This will guarantee that no directed path from $t$ to $s$ will appear in the distribution network, and vertex $s$ can avoid paying the transit duties (see \cite{ref-dellamico} for more details).

A new variation on the MCA named the \textit{Precedence-Constrained \linebreak Minimum-Cost Arborescence problem with Waiting Times} (PCMCA-WT) is introduced in this work. The problem is an extension on the PCMCA characterized by an additional constraint. {Given a spanning arborescence rooted at vertex $r$, with arc costs indicating the time required to traverse an arc, assume there is a flow which starts at the root vertex $r$ and traverses each path of the arborescence. For each precedence $(s, t) \in R$, we must guarantee that the time at which the flow enters $s$ is smaller than or equal to the time at which the flow enters $t$.} As an example, assume that $(b,a),(c,a),(d,a) \in R$, and the flow {enters} vertex $b$ at time step 5, vertex $c$ at time step 10, and vertex $d$ at time step 15. Therefore, the flow must {enter} vertex $a$ at a time step greater than or equal to 15, and if the cost of the path from $r$ to $a$ is equal to 10, then this will result in a waiting time of 5 at vertex $a$. The objective of the problem is to find an arborescence $T$ that has a minimum total cost, plus total waiting times, where the flow never {enters} $t$ earlier than {entering} $s$ for all $(s, t) \in R$.

The contributions of this paper can be summarized as:
\begin{enumerate}
	\item Introducing a scalable and efficient integer linear programming model for the PCMCA.
	\item Introducing the PCMCA-WT as a new variation of the MCA.
	\item Proving that {both the PCMCA and the PCMCA-WT are {{\sc NP}-hard}}.
\end{enumerate} 

The rest of the paper is organized as follows. Section \ref{sec-pcmca} presents a proof of complexity and a new mixed integer linear programming model (MILP) for the PCMCA. Section \ref{sec-pcmrt} presents {a proof of complexity and} several mixed integer linear programming models for the PCMCA-WT. Section \ref{sec-results} discusses computational results, while some conclusions are summarized in Section \ref{sec-conclusions}.

\begin{sloppypar}
\section{The Precedence-Constrained Minimum-Cost Arborescence Problem}
\label{sec-pcmca}
\end{sloppypar}

The \textit{Precedence-Constrained Minimum-Cost Arborescence problem} can be formally described as follows. Let $G=(V,A)$ be a directed graph, {$r \in V$}, and $P=(V,R)$ be a precedence graph. Let $c_{ij}$ be a cost associated with every arc $(i,j) \in A$. An arc $(s,t) \in R$ is a precedence relationship between the two vertices $s,t \in V$. The objective of the problem is to find a minimum-cost arborescence $T$ rooted at vertex $r \in V$ such that, for each $(s,t) \in R$, $t$ must not belong to the unique path in $T$ that connects $r$ to $s$. For simplicity, we always assume that for the root $r \in V$, $(s,r) \notin A$ for all $s \in V \backslash \{r\}$, as by definition none of these arcs would be part of an arborescence rooted at $r$, and $(s, r) \notin R$ for all $s \in V \backslash \{r\}$ as the problem would be infeasible otherwise.

Figure \ref{fig-pcmca-example} presents an example that shows the difference between the classic MCA and the PCMCA. The example instance graph with its respective arc costs is shown in the figure on the left, with the precedence relationship $(3, 1)$ highlighted in red. The figure in the middle shows a feasible MCA solution with a cost of 3. The MCA solution is infeasible for the PCMCA since $(3, 1) \in R$, and vertex 1 belongs to the directed path connecting $r$ to vertex 3. To make the solution feasible for the PCMCA, vertex 1 must succeed vertex 3 on the same directed path, or the two vertices must reside on two disjoint paths. A feasible solution with a cost of 4 is shown in the figure on the right.

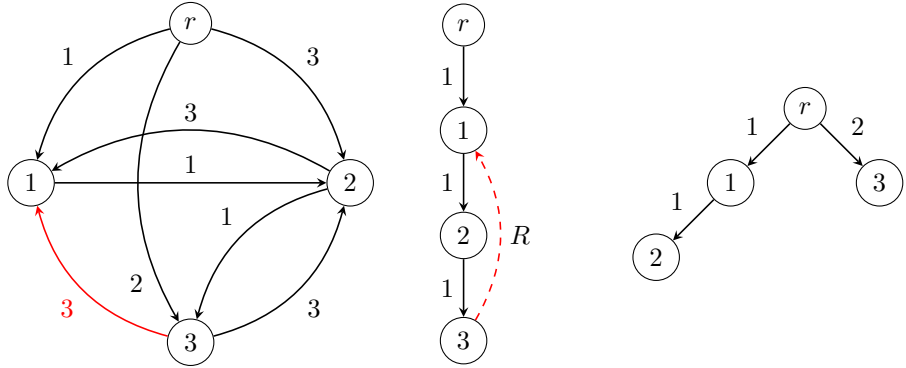
\begin{figure}[t]
	\centering
	\begin{subfigure}{.3\textwidth}
		\centering
		\begin{tikzpicture}[node distance={3cm}, main/.style = {draw, circle}]
			\node[main] (0) {$r$}; 
			\node[main] (1) [below left of=0] {1};
			\node[main] (2) [below right of=0] {2};
			\node[main] (3) [below right of=1] {3};
			
			\path[->,>=stealth, line width=0.6] (0) edge [bend right, auto=right] node {1} (1);
			\path[->,>=stealth, line width=0.6] (0) edge [bend left, auto=left] node {3} (2);
			\path[->,>=stealth, line width=0.6] (0) edge [bend right, auto=right, pos=0.8] node {2} (3);
			
			\path[->,>=stealth, line width=0.6] (1) edge [out=0,in=180, auto=left] node {1} (2);
			
			\path[->,>=stealth, line width=0.6] (2) edge [bend right, auto=right] node {3} (1);
			\path[->,>=stealth, line width=0.6] (2) edge [bend right, auto=right] node {1} (3);
			
			\path[->,>=stealth, color=red, line width=0.6] (3) edge [bend left, auto=left] node {3} (1);
			\path[->,>=stealth, line width=0.6] (3) edge [bend right, auto=right] node {3} (2);
		\end{tikzpicture}
	\end{subfigure}\hspace{9mm}%
	\begin{subfigure}{.3\textwidth}
		\centering
		\begin{tikzpicture}[node distance={1.4cm}, main/.style = {draw, circle}]
			\node[main] (0) {$r$}; 
			\node[main] (1) [below of=0] {1};
			\node[main] (2) [below of=1] {2};
			\node[main] (3) [below of=2] {3};
			
			\path[->,>=stealth, line width=0.6] (0) edge [auto=right] node {1} (1);
			\path[->,>=stealth, line width=0.6] (1) edge [auto=right] node {1} (2);
			\path[->,>=stealth, line width=0.6] (2) edge [auto=right] node {1} (3);
			\path[dashed, ->,>=stealth, color=red, line width=0.6] (3) edge [bend right, auto=right, line width=0.6] node {\textcolor{black}{$R$}} (1);
		\end{tikzpicture}
	\end{subfigure}
	\begin{subfigure}{.3\textwidth}
		\centering
		\begin{tikzpicture}[node distance={1.4cm}, main/.style = {draw, circle}]
			\node[main] (0) {$r$}; 
			\node[main] (1) [below left of=0] {1};
			\node[main] (2) [below left of=1] {2};
			\node[main] (3) [below right of=0] {3};
			
			\path[->,>=stealth, line width=0.6] (0) edge [auto=right] node {1} (1);
			\path[->,>=stealth, line width=0.6] (0) edge [auto=left] node {2} (3);
			\path[->,>=stealth, line width=0.6] (1) edge [auto=right] node {1} (2); node {1} (3);
		\end{tikzpicture}
	\end{subfigure}
	\caption{Comparing a MCA and a PCMCA solution. The graph on the left shows the instance graph with its respective arc costs, with the precedence relationship $(3, 1) \in R$ highlighted in red. The graph in the middle shows the optimal MCA, and the graph on the right shows the optimal PCMCA. The MCA solution is not a feasible PCMCA solution since vertex 1 precedes vertex 3 on the same directed path and $(3, 1) \in R$.}
	\label{fig-pcmca-example}
\end{figure}

\vspace{1em}
\subsection{Computational Complexity}
\label{sec-pcmca-complexity}

Some of the Minimum-Cost Arborescence variations mentioned in Section \ref{sec-introduction} belong to the {{\sc NP}-hard} complexity class. In this section we show that the Precedence-Constrained Minimum-Cost Arborescence Problem is also {{\sc NP}-hard}. The proof is inspired by {the one} introduced in \cite{ref-avoiding-pairs} for the \textit{Path Avoiding Forbidden Pairs problem}.

\begin{thm}
	The PCMCA is {{\sc NP}-\textup{hard.}}
	\label{thm-1}
\end{thm}

\begin{figure}[t]
	\centering
	\begin{tikzpicture}[scale=1,transform shape]
		\node[shape=circle,draw=black] (r) at (0,0) {$r$};
		
		\node[shape=circle,draw=black] (s) at (1.5,0) {$s$};
		
		\node[shape=circle,draw=black] (sp) at (5.5,-3) {$s'$};
		
		\node[shape=circle,draw=black] (v11) at (3,1) {$v_{11}$};
		\node[shape=circle,draw=black] (v12) at (3,0) {$v_{12}$};
		\node[shape=circle,draw=black] (v13) at (3,-1) {$v_{13}$};
		
		\node[shape=circle,draw=white] (e11) at (5,1) {};
		\node[shape=circle,draw=white] (e12) at (5,0) {};
		\node[shape=circle,draw=white] (e13) at (5,-1) {};
		
		\node[shape=circle,draw=white] (d1) at (5.5,1) {$\dots$};
		\node[shape=circle,draw=white] (d2) at (5.5,0) {$\dots$};
		\node[shape=circle,draw=white] (d3) at (5.5,-1) {$\dots$};
		
		\node[shape=circle,draw=white] (e21) at (6,1) {};
		\node[shape=circle,draw=white] (e22) at (6,0) {};
		\node[shape=circle,draw=white] (e23) at (6,-1) {};
		
		\node[shape=circle,draw=black] (vm1) at (8.5,1) {$v_{m1}$};
		\node[shape=circle,draw=black] (vm2) at (8.5,0) {$v_{m2}$};
		\node[shape=circle,draw=black] (vm3) at (8.5,-1) {$v_{m3}$};
		
		\node[shape=circle,draw=black] (t) at (10.5,0) {$t$};
		
		\path [->,>=stealth, line width=0.6] (r) edge node[left] {} (s);
		
		\path [->,>=stealth, line width=0.6] (s) edge node[left] {} (v11);
		\path [->,>=stealth, line width=0.6] (s) edge node[left] {} (v12);
		\path [->,>=stealth, line width=0.6] (s) edge node[left] {} (v13);
		
		\path [->,>=stealth, line width=0.6] (v11) edge node[left] {} (e11);
		\path [->,>=stealth, line width=0.6] (v11) edge node[left] {} (e12);
		\path [->,>=stealth, line width=0.6] (v11) edge node[left] {} (e13);
		
		\path [->,>=stealth, line width=0.6] (v12) edge node[left] {} (e11);
		\path [->,>=stealth, line width=0.6] (v12) edge node[left] {} (e12);
		\path [->,>=stealth, line width=0.6] (v12) edge node[left] {} (e13);
		
		\path [->,>=stealth, line width=0.6] (v13) edge node[left] {} (e11);
		\path [->,>=stealth, line width=0.6] (v13) edge node[left] {} (e12);
		\path [->,>=stealth, line width=0.6] (v13) edge node[left] {} (e13);
		
		\path [->,>=stealth, line width=0.6] (e21) edge node[left] {} (vm1);
		\path [->,>=stealth, line width=0.6] (e21) edge node[left] {} (vm2);
		\path [->,>=stealth, line width=0.6] (e21) edge node[left] {} (vm3);
		
		\path [->,>=stealth, line width=0.6] (e22) edge node[left] {} (vm1);
		\path [->,>=stealth, line width=0.6] (e22) edge node[left] {} (vm2);
		\path [->,>=stealth, line width=0.6] (e22) edge node[left] {} (vm3);
		
		\path [->,>=stealth, line width=0.6] (e23) edge node[left] {} (vm1);
		\path [->,>=stealth, line width=0.6] (e23) edge node[left] {} (vm2);
		\path [->,>=stealth, line width=0.6] (e23) edge node[left] {} (vm3);
		
		\path [->,>=stealth, line width=0.6] (vm1) edge node[left] {} (t);
		\path [->,>=stealth, line width=0.6] (vm2) edge node[left] {} (t);
		\path [->,>=stealth, line width=0.6] (vm3) edge node[left] {} (t);
		
		\path [->,>=stealth, line width=0.6] (r) edge node[left] {} (sp);
		
		\path [->,>=stealth, line width=0.6] (sp) edge node[left] {} (v11);
		\path [->,>=stealth, line width=0.6] (sp) edge node[left] {} (v12);
		\path [->,>=stealth, line width=0.6] (sp) edge node[left] {} (v13);
		
		\path [->,>=stealth, line width=0.6] (sp) edge node[left] {} (vm1);
		\path [->,>=stealth, line width=0.6] (sp) edge node[left] {} (vm2);
		\path [->,>=stealth, line width=0.6] (sp) edge node[left] {} (vm3);
		
		\draw [dashed,->, line width=0.6] (t) to [out=270,in=-10] (sp);
		\draw [dashed,->, line width=0.6] (vm1) to [out=150,in=30] (v11);
	\end{tikzpicture}
	\caption{A PCMCA instance reduced from {\sc 3-SAT}. The set of vertices $v_{ij} \subset V$ are the literals of the {\sc 3-SAT} problem, where each layer is a clause that is completely connected only to the clause in the next layer. Dashed arcs show the precedence constraints $R$ between pairs of vertices. As an example, we assume literal $v_{11}$ is the negation of $v_{m1}$, therefore there is a precedence relationship $(v_{m1}, v_{11}) \in R$ which will enforce that no literal and its negation can belong to the same path.}
	\label{reduction}
\end{figure}
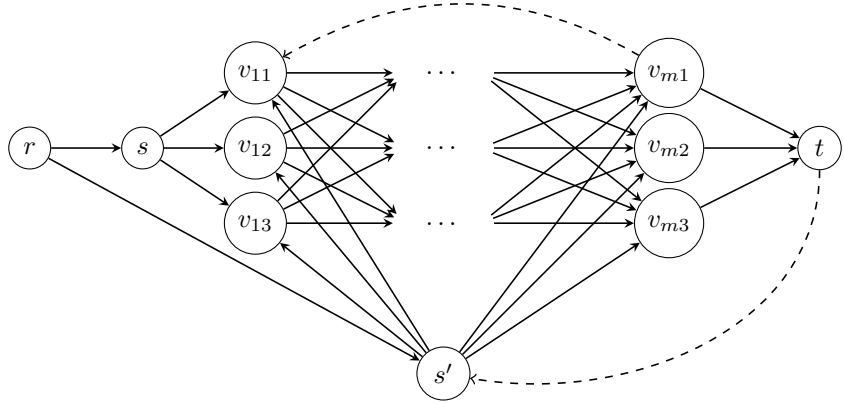

\begin{proof}
	By reduction from {\sc 3-SAT}: {Let $X=\{x_1,x_2,\dots,x_t\}$ be a set of variables}. Let $\Phi = C_1 \, \wedge \, C_2 \, \wedge \, \dots \wedge \, C_m$ be a boolean expression in 3-conjunctive normal form, such that each clause $C_i$, $i = 1, \dots,m$, {is denoted by $(v_{i1} \vee v_{i2} \vee v_{i3})$, where each literal $v_{ik}, 1 \leq k \leq 3$, is associated to one variable in $X$ or its negation.} We will construct a graph $G$ and a set of precedence constraints $R$ such that there exists a {feasible solution of the PCMCA problem} in $G$ if and only if $\Phi$ is satisfiable.
	
	Let $G = (V, A)$ where $V = \{r\} \cup \{s\} \cup \{s'\} \cup \{t\} \cup C$, with $C$, $A$ and $R$ defined as follows.
	\begin{align*}
		C &= {\{v_{ik} : 1 \leq i \leq m, 1 \leq k \leq 3\}} \\
		A &=\{(r,s),(r,s')\} \cup\{(s,v_{1j}),1\le j\le3)\} \cup\{(v_{mj},t),1\le j\le 3\} \\ 
		& \cup \{(v_{ij},v_{i+1,k}),1\le i<m, 1\le j,k\le 3\} \cup \{(s',v_{ij}),1\le i\le m, 1\le j\le 3\} \\
		R &= \{(t,s')\} \cup\{(v_{hk},v_{ij}): h>i, \, v_{hk} \text{ and } v_{ij} \text{ refer to the same variable, but} \\ & \quad\,\,\, \text{exactly one of the two literals is negated} \}
	\end{align*}
	
	Note that $C$ contains $3m$ vertices, one for each literal of each clause $C_i$, with all arcs having {an equal positive cost}. The three sets $C$, $A$, and $R$ induce the graph shown in Figure \ref{reduction}. The set of precedence constraints, {besides $(t,s')$,} is between two vertices that refer to the same literal, but exactly one of the two literals is negated. If a {feasible solution $T$} of the PCMCA problem can be found in $G$, this implies that:
	
	\begin{enumerate}
		\item no path from $s'$ to $t$ exists {in $T$}
		\item in any (rooted) path there is no pair of vertices corresponding to a variable and its negation
		\item there is a unique path $P$ from $r$ to $t$ which passes through $s$ and through a vertex of each clause
	\end{enumerate}
	
	{The formula can be satisfied by assigning true values to all the literals corresponding to the vertices in $P \cap C$, and assigning false values {to all the variables {not} associated with these literals}. This satisfies all the clauses.}
	
	{Conversely, if} the formula is satisfied then each clause has at least one literal with true value, and no variable is assigned to both true and false (in different clauses). We construct a PCMCA feasible solution as follows. We start by building a path $P$ from $r$ to $t$ which includes $s$ and exactly one vertex from each clause, corresponding to a literal with true value. We complete the arborescence by adding $(r,s')$ and $(s',v)$ for each $v\not\in P$. 
\end{proof}

\subsection{A Set-Based Model}
\label{sec-pcmca-model}

The MILP previously proposed in \cite{ref-dellamico} suffers from scalability issues because of the cubic number of variables (relative to the number of vertices and the precedence relationships) used to model the precedence relationships between vertex pairs. A new integer linear program for the PCMCA is introduced in this section. 

We extend the classic connectivity constraints for the MCA \cite{ref-houndji} in such a way to take precedences into account. When considering a set $S \subseteq V \backslash \{r\}$ we add a constraint for each $j \in S${,} and we force that at least one active arc must enter $S$ coming from {the set} of vertices allowed to precede $j$ on the path connecting $j$ to $r$.

Let $x_{ij}$ be a variable associated with every arc $(i,j) \in A$ such that $x_{ij} = 1$ if $(i,j) \in T$ and 0 otherwise, where $T$ is the resulting optimal arborescence. Let $V_j = \{i \in V : (j,i) \notin R\}$ be the set of vertices that can precede $j$ on a directed path from the root without introducing precedence violations {or, in turn, a violating path, which is a directed path that violates some precedence relationship in $R$}.

\allowdisplaybreaks
\begin{align}
	\text{minimize} & \sum_{(i,j) \in A} c_{ij} x_{ij}
	\label{eqn-1} 
	\\
	\text{subject to} & \sum_{(i,j) \in A} x_{ij} = 1 & \forall j \in V \backslash \{r\}
	\label{eqn-2}
	\\
	& \sum_{\substack{(i,k) \in A:\\ i \in V_j \backslash S, \, k \in S}} x_{ik} \geq 1 & \forall j \in V \backslash \{r\}, \forall S \subseteq V_j \backslash \{r\} : j \in S
	\label{eqn-3}
	\\
	& x_{ij} \in \{0,1\} & \forall (i,j) \in A
	\label{eqn-4}
\end{align}

Constraints (\ref{eqn-2}) implies the first property of an arborescence namely that every vertex $v \in V \backslash \{r\}$ must have a single parent. Constraints (\ref{eqn-3}) model the connectivity constraints, that is every vertex $j \in V \backslash \{r\}$ must be reachable from the root. Note that $V_j \backslash S$ contains at least $r$, while $S$ contains at least $j$. The set of constraints (\ref{eqn-3}) reduces to the classical connectivity constraints for the MCA which are $\sum_{(i,k) \in A: \, i \notin S, \, k \in S} x_{ik} \geq 1$ $\forall S \subseteq V \backslash \{r\}$ when the set $R$ of precedence relationships is empty. This is because when $R$ is an empty set, $V_j = V$ for all $j \in V\backslash\{r\}$. Constraints (\ref{eqn-3}) also impose the precedence relationships. Inequality (\ref{eqn-3}) implies that the resulting arborescence will not include vertex $t$ in the directed path connecting $r$ to $s$ when $(s,t) \in R$. Note that this is the same inequality named weak $\sigma$-inequality considered by Ascheuer, Jünger \& Reinelt \cite{ref-ascheuer} for the Sequential Ordering Problem. Finally, constraints (\ref{eqn-4}) define the domain of the variables. The MILP model proposed has $O(|A|)$ variables, and $O(|A|)$ constraints, plus an exponential number of connectivity constraint (\ref{eqn-3}). Although the number of constraints of the \textit{Set-Based Model} is exponential, it is more efficient at solving the problem than the model introduced in \cite{ref-dellamico}. This is because in practice the number of constraints that are dynamically added to the model is small. Moreover, the model uses a smaller number of variables. An experimental validation for these considerations will be provided by the experiments in Section \ref{sec-results}.

One approach to solve a {linear relaxation model (LR-model)} that has an exponential number of constraints is to start by solving the {LR-model} without including a large set of constraints (such as (\ref{eqn-3})), then iteratively adding a constraint once it is violated, then solving the new {LR-model} again. {A procedure for finding a violated constraint is called a \emph{separation procedure}.} The optimal solution {of the LR-model} is found as soon as there are no violated constraints. {When solving a MILP, the optimality gap needs to be closed to find the optimal solution even if no violated inequality is found.} 

{In the literature, the large set of constraints that are necessary to model the problem, but are added dynamically to the model only when they are violated, are known as \textit{lazy constraints}. Note that using this approach, the separation procedure must also be used to check the feasibility of integer solutions found by the linear relaxation or primal heuristics.}

\begin{algorithm}[H]
	\begin{algorithmic}[1]
		\State $G = (V,A)$ is a directed graph, $\bar{x}$ is a fractional LP-solution
		\Procedure {Find\_Violated\_Inequality} {$G$, $\bar{x}$}
		\For {$j \in V \backslash\{r\}$}
		\State Construct a directed graph ${D_j} = (V_j, A')$ such that:
		\State $V_j = \{i \in V : (j,i) \notin R\}$
		\State $A' = \{(i,k) \in A \, | \, i,k \in V_j\}$
		\State $c_{ik} = \bar{x}_{ik} \,\,\,\, \forall \, (i,k) \in A'$
		\State Calculate a minimum $(r, j)$-cut $C$ in ${D_j}$
		\If{$\text{the cost of } C < 1$}
		\State return the violated inequality $\sum_{(i,k) \in C \,\,} x_{ik} \geq 1$
		\EndIf
		\EndFor
		\EndProcedure
		\caption{Separation Procedure for Inequalities (\ref{eqn-3})}\label{alg-1}
	\end{algorithmic}
\end{algorithm}

{Algorithm \ref{alg-1} describes the separation procedure for inequalities (\ref{eqn-3}). Let $\bar{x}$ be a solution of the linear relaxation or a candidate primal solution. An inequality (\ref{eqn-3}) that is violated by the solution $\bar{x}$ can be detected by computing a minimum $(r,j)$-cut $C$ in a directed graph $D_j = (V_j, A')$, where $A'$ is equal to the set of arcs $A$ minus the arcs incident to the immediate successors of $j$ in the precedence graph. The cost $c_{ik}$ of an arc $(i,k) \in A'$ is equal to $\bar{x}_{ik}$. The value of the minimum $(r,j)$-cut $C$ in $D_j$ can tell us the following about the given fractional solution:}

\begin{enumerate}
	\item If the cost of a minimum cut is equal to 0, {then  vertex $j$ is not reachable from $r$ {in $D_j$}. In this case, the solution does not contain a path from $r$ to $j$, or contains a single or multiple paths from $r$ to $j$, all of which pass through a successor of $j$.}
	\item If the cost of a minimum cut is in the range $(0, 1)$, {then vertex $j$ is reachable from $r$ {in $D_j$}. In this case, the solution contains multiple paths from $r$ to $j$, and at least one of them passes through a successor of $j$.}
	\item If the cost of a minimum cut is equal to 1, then {vertex $j$ is reachable from $r$ through a single or multiple paths {in $D_j$}, although possibly some of them pass through a successor of $j$.}
\end{enumerate}

{In the first two cases, the minimum cut $C$ defines an inequality (\ref{eqn-3}) violated by $\bar{x}$, however in the last case a violated inequality (\ref{eqn-3}) {does not exist} even if the fractional solution $\bar{x}$ contains a violating path. Therefore, although inequalities (\ref{eqn-3}) are valid inequalities for that PCMCA, there are fractional LP-solutions that contain violating paths, but satisfy inequalities (\ref{eqn-3}).}

Figure \ref{fig-fail-fractional} shows an example on how the {separation} procedure works. The figure also shows a fractional solution {of the \emph{Set-Based} model} that contains a violating path, but does not violate any inequality (\ref{eqn-3}). {Another example, showing how the \emph{Path-Based} model introduced in  \cite{ref-dellamico} fails to detect a violating path in a fractional solution, is presented in \ref{a1}.}

A drawback of the \textit{Set-Based} model is the high computational complexity of the separation procedure of inequalities (\ref{eqn-3}), which has a complexity of $O(n^4)$, assuming it uses an $O(n^3)$ algorithm for computing a minimum $(s,t)$-cut in {$D_j$} \cite{ref-cut}. 

\begin{figure}[t]
	\centering
	\begin{subfigure}{.45\textwidth}
		\centering
		\begin{tikzpicture}[node distance={2.5cm}, main/.style = {draw, circle},scale=0.9, transform shape]
			\node[main] (r) {$r$}; 
			\node[main] (t) [below left of=0] {$t$};
			\node[main] (1) [below left of=t] {1};
			\node[main] (2) [right of=1] {2};
			\node[main] (3) [right of=2] {3};
			\node[main] (s) [below of=2] {$s$};
			
			\path[->,>=stealth, line width=0.6] (r) edge [auto=right] node {1.0} (t);
			\path[->,>=stealth, line width=0.6] (r) edge [auto=left] node {0.5} (2);
			\path[->,>=stealth, line width=0.6] (r) edge [auto=left] node {1.0} (3);
			
			\path[->,>=stealth, line width=0.6] (t) edge [auto=right] node {0.5} (1);
			\path[->,>=stealth, line width=0.6] (t) edge [auto=right] node {0.5} (2);
			
			\path[->,>=stealth, line width=0.6] (2) edge [auto=left] node {0.5} (1);
			
			\path[->,>=stealth, line width=0.6] (1) edge [auto=right] node {0.5} (s);
			\path[->,>=stealth, line width=0.6] (3) edge [auto=left] node {0.5} (s);
		\end{tikzpicture}
	\end{subfigure}
	\begin{subfigure}{.45\textwidth}
		\centering
		\begin{tikzpicture}[node distance={2.5cm}, main/.style = {draw, circle},scale=0.9, transform shape]
			\node[main] (r) {$r$}; 
			\node[main] (1) [below left of=t] {1};
			\node[main] (2) [right of=1] {2};
			\node[main] (3) [right of=2] {3};
			\node[main] (s) [below of=2] {$s$};
			
			\path[->,>=stealth, line width=0.6] (r) edge [auto=right] node {0.5} (2);
			\path[->,>=stealth, line width=0.6] (r) edge [auto=left] node {1.0} (3);
			
			\path[->,>=stealth, line width=0.6] (2) edge [auto=right] node {0.5} (1);
			
			\path[->,>=stealth, line width=0.6] (1) edge [auto=right] node {0.5} (s);
			\path[->,>=stealth, line width=0.6] (3) edge [auto=left] node {0.5} (s);
			
			\draw [dashed,red, line width=0.6] (-3.5,-5) to[out=45,in=140] (1.5, -5);
		\end{tikzpicture}
	\end{subfigure}
	\caption{An example of a fractional solution {of the \emph{Set-Based} model} that contains a violating path, and does not violate an inequality (\ref{eqn-3}). Every arc cost is associated with the value of its respective variable $x_{ij}$. For this solution we have the violated precedence $(s, t) \in R$. The figure on the left shows the solution, and the figure on the right shows the graph ${D_j}$ which has a minimum $(r, s)$-cut of value 1 indicated by the red dashed line.}
	\label{fig-fail-fractional}
\end{figure}
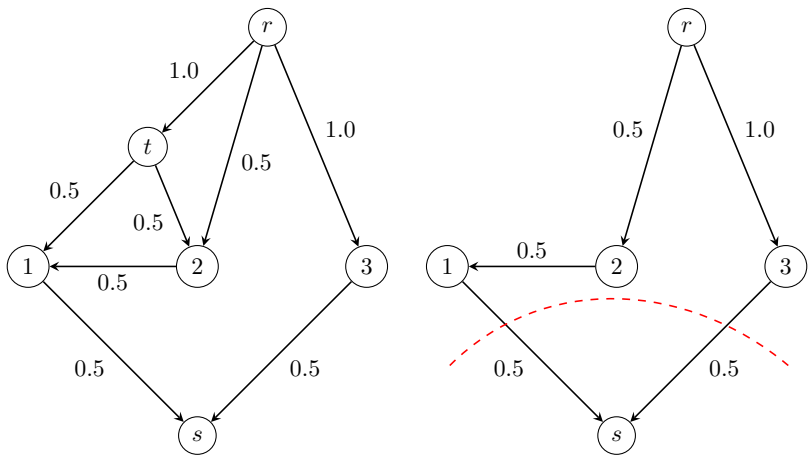

\begin{sloppypar}
\section{The Precedence-Constrained Minimum-Cost Arborescence Problem with Waiting Times}
\label{sec-pcmrt}
\end{sloppypar}

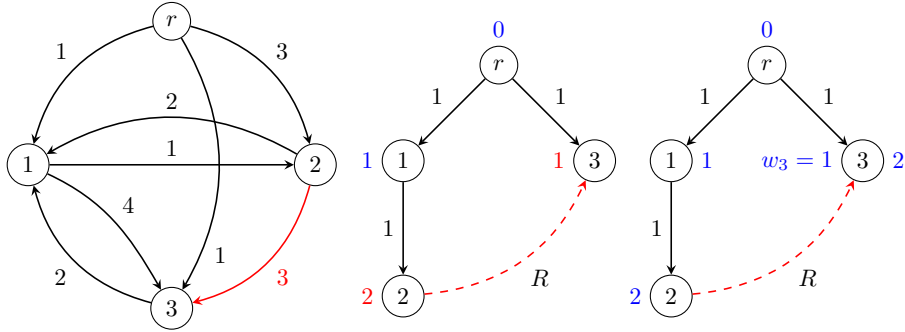
\begin{figure}[t]
	\centering
	\hspace{-4mm}%
	\begin{subfigure}{.3\textwidth}
		\centering
		\begin{tikzpicture}[node distance={3cm}, main/.style = {draw, circle},scale=0.9, transform shape]
			\node[main] (0) {$r$}; 
			\node[main] (1) [below left of=0] {1};
			\node[main] (2) [below right of=0] {2};
			\node[main] (3) [below right of=1] {3};
			
			\path[->,>=stealth, line width=0.6] (0) edge [bend right, auto=right] node {1} (1);
			\path[->,>=stealth, line width=0.6] (0) edge [bend left, auto=left] node {3} (2);
			\path[->,>=stealth, line width=0.6] (0) edge [bend left, auto=left, pos=0.8] node {1} (3);
			
			\path[->,>=stealth, line width=0.6] (1) edge [out=0,in=180, auto=left] node {1} (2);
			\path[->,>=stealth] (1) edge [out=-25,in=120, auto=left, line width=0.6] node {4} (3);
			
			\path[->,>=stealth, line width=0.6] (2) edge [bend right, auto=right] node {2} (1);
			\path[->,>=stealth, color=red, line width=0.6] (2) edge [bend left, auto=left] node {3} (3);
			
			\path[->,>=stealth, line width=0.6] (3) edge [bend left, auto=left] node {2} (1);
		\end{tikzpicture}
	\end{subfigure}\hspace{9mm}%
	\begin{subfigure}{.3\textwidth}
		\centering
		\begin{tikzpicture}[node distance={2cm}, main/.style = {draw, circle},scale=0.9, transform shape]
			\node[main, label=above:\textcolor{blue}{0}] (0) {$r$}; 
			\node[main, label=left:\textcolor{blue}{1}] (1) [below left of=0] {1};
			\node[main, label=left:\textcolor{red}{2}] (2) [below of=1] {2};
			\node[main, label=left:\textcolor{red}{1}] (3) [below right of=0] {3};
			
			\path[->,>=stealth, line width=0.6] (0) edge [auto=right] node {1} (1);
			\path[->,>=stealth, line width=0.6] (0) edge [auto=left] node {1} (3);
			
			\path[->,>=stealth, line width=0.6] (1) edge [auto=right] node {1} (2);
			
			\path[dashed, ->,>=stealth, color=red, line width=0.6] (2) edge [bend right, auto=right] node {\textcolor{black}{$R$}} (3);
		\end{tikzpicture}
	\end{subfigure}
	\begin{subfigure}{.3\textwidth}
		\centering
		\begin{tikzpicture}[node distance={2cm}, main/.style = {draw, circle},scale=0.9, transform shape]
			\node[main, label=above:\textcolor{blue}{0}] (0) {$r$}; 
			\node[main, label=right:\textcolor{blue}{1}] (1) [below left of=0] {1};
			\node[main, label=left:\textcolor{blue}{2}] (2) [below of=1] {2};
			\node[main, label=right:\textcolor{blue}{2}, label=left: \textcolor{blue}{$w_3 = 1$}] (3) [below right of=0] {3};
			
			\path[->,>=stealth, line width=0.6] (0) edge [auto=right] node {1} (1);
			\path[->,>=stealth, line width=0.6] (0) edge [auto=left] node {1} (3);
			
			\path[->,>=stealth, line width=0.6] (1) edge [auto=right] node {1} (2);
			
			\path[dashed, ->,>=stealth, color=red, line width=0.6] (2) edge [bend right, auto=right] node {\textcolor{black}{$R$}} (3);
		\end{tikzpicture}
	\end{subfigure}
	\caption{Comparing an instance solved as a PCMCA, and solved as a PCMCA-WT. The graph on the left shows the instance with its respective arc costs, and the precedence relationship $(2, 3) \in R$ highlighted in red. The graph in the middle shows the optimal PCMCA {solution of cost 3}, and the graph on the right shows the optimal PCMCA-WT {solution of cost 4}. The PCMCA solution is not a feasible PCMCA-WT solution since the flow {enters} vertex 3 before {entering} vertex 2 and $(2, 3) \in R$.}
	\label{fig-pcmcawt-example}
\end{figure}

In the \textit{Precedence-Constrained Minimum-Cost Arborescence Problem with Waiting Times} (PCMCA-WT) {a flow starts from the root $r$ at time 0, and traverses each path of the arborescence. The cost $c_{ij}$ of an arc $(i, j) \in A$ represents the time required to traverse that arc. Let $d_j$ be the time at which the flow enters vertex $j \in V$. For any $(s,t) \in R$, $d_t \geq d_s$, which means that the flow must enter vertex $t$ at the same time step or after entering vertex $s$. Let $w_j$ be the waiting time before the flow enters vertex $j$ required to respect the aforementioned constraint. The objective is to find an arborescence $T$ that has a minimum total cost plus total waiting time, where the flow never enters $t$ earlier than entering $s$ for all $(s, t) \in R$. For simplicity, we always assume that for the root $r \in V$, $(i, r) \notin A$ for all $i \in V \backslash \{r\}$, as by definition none of these arcs would be part of an arborescence rooted at $r$, and $(s, r) \notin R$ for all $s \in V \backslash \{r\}$, as the problem would be infeasible otherwise.}

Figure \ref{fig-pcmcawt-example} presents an example that shows the difference between the PCMCA and the PCMCA-WT. Next to each vertex we have its corresponding $d_t$ value. In this example we have the precedence relationship $(2, 3)$ highlighted in red. The two solutions depicted are valid solutions for the PCMCA, since they both satisfy the precedence constraints, that is $t$ never precede $s$ on the same directed path for all $(s,t) \in R$. The solution in the middle shows the optimal PCMCA {solution} with a total cost of 3 (sum of all the arcs). We can see that the solution in the middle is not a feasible PCMCA-WT solution since $(2, 3) \in R$ but $d_3 < d_2$. The solution on the right shows an optimal PCMCA-WT solution with a cost of 4 (sum of all the arcs plus waiting time at each vertex). The solution results in a waiting time of 1 at vertex 3, since the time from $r$ to $2$ is 2, and the time from $r$ to $3$ is 1.

\vspace{1em}
\subsection{Computational Complexity}
\label{sec-pcmrt-complexity}

{In this section we show that the PCMCA-WT is {\sc NP}-hard.}

{The \textit{Rectilinear Steiner Arborescence} (RSA) \textit{Problem} \cite{ref-rsa} is an {\sc NP}-hard problem formally defined as follows. Let $P = \{p_1, p_2, \dots, p_n\}$ be a set of points in the first quadrant of the Cartesian plane, where $p_i = (x_i, y_i)$ with $x_i,y_i \geq 0$, and $p_1=(0,0)$.} {A complete grid can be created, where the points in P are on the intersections of vertical and horizontal lines. A set $S$ of Steiner vertices can be added, corresponding to the $O(|P|^2)$ intersection points not overlapping with the points in $P$.  The arcs of the problem are the right-directed horizontal segments and the up-directed vertical segments between two {adjacent} points of the grid $P \cup S$. {The cost associated with each arc $(p_i, p_j)$ is defined as $|x_i - x_j| + |y_i - y_j|$.} Figure \ref{rsa-example} shows an example of an RSA instance with 5 points, and the relative Steiner vertices.}

{Given a positive value $k$, the decision version of the RSA problem consists in {deciding whether there is} an arborescence with total {length} not greater than $k$ such that the arborescence is rooted at $p_1$ and it contains a unique path from $p_1$ to $p_i$ for all $i \in \{1,2,\dots,{n}\}$. Note that the length of each path from $p_1$ to  $p_i$ is $x_i + y_i$ by construction.}

\begin{figure}[t]
	\centering
	\hspace{-4mm}%
		\centering
		\begin{tikzpicture}[node distance={3cm}, main/.style = {draw, circle},scale=0.9, transform shape]
			\node[main] at (0,0) (p1) {$p_1$}; 
			\node[main] at (2.6, 0) (p2) {$p_2$};
			\node[main] at (3.9, 1.6) (p3) {$p_3$};			
			\node[main] at (6.3, 5) (p4) {$p_4$};			
			\node[main] at (1.3, 5) (p5) {$p_5$};
			
			\node[main] at (0, 1.6) (s1) {$s_1$};	
			\node[main] at (0, 5) (s2) {$s_2$};
			\node[main] at (1.3, 0) (s3) {$s_3$};	
			\node[main] at (1.3, 1.6) (s4) {$s_4$};
			\node[main] at (2.6, 1.6) (s5) {$s_5$};
			\node[main] at (2.6, 5) (s6) {$s_6$};	
			\node[main] at (3.9, 0) (s7) {$s_7$};	
			\node[main] at (3.9, 5) (s8) {$s_8$};
			\node[main] at (6.3, 0) (s9) {$s_9$};
			\node[main, scale=0.9] at (6.3, 1.6) (s10) {$s_{10}$};						
			\path[->,>=stealth, dashed] (p1) edge (s3);
			\path[->,>=stealth, dashed] (s3) edge (p2);
			\path[->,>=stealth, dashed] (p2) edge (s7);
			\path[->,>=stealth, dashed] (s7) edge (s9);
			\path[->,>=stealth, dashed] (s1) edge (s4);
			\path[->,>=stealth, dashed] (s4) edge (s5);
			\path[->,>=stealth, dashed] (s5) edge (p3);
			\path[->,>=stealth, dashed] (p3) edge (s10);
			\path[->,>=stealth, dashed] (s2) edge (p5);
			\path[->,>=stealth, dashed] (p5) edge (s6);
			\path[->,>=stealth, dashed] (s6) edge (s8);
			\path[->,>=stealth, dashed] (s8) edge (p4);
			
			\path[->,>=stealth, dashed] (p1) edge (s1);
			\path[->,>=stealth, dashed] (s1) edge (s2);
			\path[->,>=stealth, dashed] (s3) edge (s4);
			\path[->,>=stealth, dashed] (s4) edge (p5);
			\path[->,>=stealth, dashed] (p2) edge (s5);
			\path[->,>=stealth, dashed] (s5) edge (s6);
			\path[->,>=stealth, dashed] (s7) edge (p3);
			\path[->,>=stealth, dashed] (p3) edge (s8);
			\path[->,>=stealth, dashed] (s9) edge (s10);
			\path[->,>=stealth, dashed] (s10) edge (p4);
			
		\end{tikzpicture}
	
	\caption{{The figure shows an RSA instance with 5 points and 10 Steiner vertices, while the dashed lines represent the arcs of the instance.}}
	\label{rsa-example}
\end{figure}
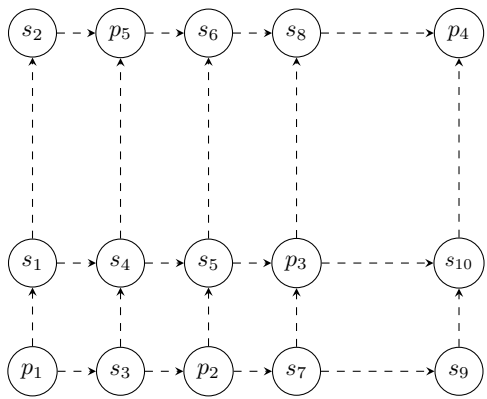

\begin{thm}
	{The PCMCA-WT is {\sc NP}-\textup{hard.}}
\end{thm}

\begin{proof}
	{By a reduction from the decision version of the RSA problem{: we} construct a graph $G=(V, A)$ and a set $R$  of precedence constraints such that there exist a PCMCA-WT solution of cost at most $k$ if and only if a RSA of cost at most $k$ exists. Given an instance of the RSA problem with a set of points P and a set of Steiner points S, consider the PCMCA-WT instance defined as follows:}
	\begin{align*}
		{V} \mkern150mu & \mkern-100mu {= \, P \cup S}  \\
		{A'} \mkern150mu & \mkern-100mu {= \{(i,j): \text{$j$ is immediately on the top of $i$ in the grid, or $j$ is im-}}\\
		& \mkern-100mu {\text{\qquad mediately on the right of $i$ in the grid\}}} \\
		{A} \mkern155mu & \mkern-100mu {= \, A' \cup \{(P_{FAR}, s_i), s_i\in S\} \text{, with } {P_{FAR} \in \argmax_{p_i \in P} \{x_i + y_i\}} \mkern100mu} \\
		{R} \mkern155mu & \mkern-100mu {= \, \{(p,P_{FAR}): p \in P \setminus \{P_{FAR}\}\}} \\	
		{c_{ij}} \mkern150mu & \mkern-100mu {= \, (x_j - x_i) + (y_j - y_i) \text{ for } (i,j) \in A'} \\
		{c_{P_{FAR}, s_i}} \mkern110mu & \mkern-100mu {= \, 0 \text{ for } s_i \in S}
	\end{align*}
	\begin{sloppypar}
		
		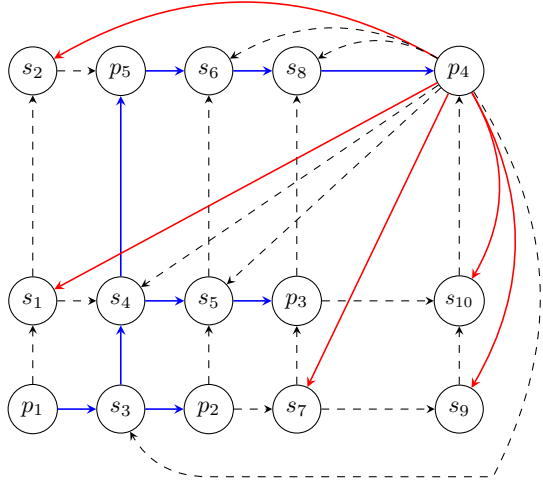
\begin{figure*}[t!]
			\centering
			\begin{tikzpicture}[node distance={3cm}, main/.style = {draw, circle},scale=0.9, transform shape]
				\node[main] at (0,0) (p1) {$p_1$}; 
				\node[main] at (2.6, 0) (p2) {$p_2$};
				\node[main] at (3.9, 1.6) (p3) {$p_3$};			
				\node[main] at (6.3, 5) (p4) {$p_4$};			
				\node[main] at (1.3, 5) (p5) {$p_5$};
				
				\node[main] at (0, 1.6) (s1) {$s_1$};	
				\node[main] at (0, 5) (s2) {$s_2$};
				\node[main] at (1.3, 0) (s3) {$s_3$};	
				\node[main] at (1.3, 1.6) (s4) {$s_4$};
				\node[main] at (2.6, 1.6) (s5) {$s_5$};
				\node[main] at (2.6, 5) (s6) {$s_6$};	
				\node[main] at (3.9, 0) (s7) {$s_7$};	
				\node[main] at (3.9, 5) (s8) {$s_8$};
				\node[main] at (6.3, 0) (s9) {$s_9$};
				\node[main, scale=0.9] at (6.3, 1.6) (s10) {$s_{10}$};		
								
				\path[->,>=stealth, color=blue, line width=0.6] (p1) edge (s3);
				\path[->,>=stealth, color=blue, line width=0.6] (s3) edge (p2);
				\path[->,>=stealth, dashed] (s7) edge (s9);
				\path[->,>=stealth, dashed] (s1) edge (s4);
				\path[->,>=stealth, dashed] (p3) edge (s10);
				\path[->,>=stealth, dashed] (s2) edge (p5);
				\path[->,>=stealth, color=blue, line width=0.6] (p5) edge (s6);
				\path[->,>=stealth, color=blue, line width=0.6] (s6) edge (s8);
				\path[->,>=stealth, color=blue, line width=0.6] (s8) edge (p4);
				\path[->,>=stealth, color=blue, line width=0.6] (s4) edge (s5);
				\path[->,>=stealth, color=blue, line width=0.6] (s5) edge (p3);	
							
				\path[->,>=stealth, dashed] (s7) edge (p3);
				\path[->,>=stealth, dashed] (p2) edge (s7);
				\path[->,>=stealth, dashed] (p1) edge (s1);
				\path[->,>=stealth, dashed] (s1) edge (s2);
				\path[->,>=stealth, color=blue, line width=0.6] (s3) edge (s4);
				\path[->,>=stealth, color=blue, line width=0.6] (s4) edge (p5);
				\path[->,>=stealth, dashed] (p2) edge (s5);
				\path[->,>=stealth, dashed] (s5) edge (s6);
				\path[->,>=stealth, dashed] (p3) edge (s8);
				\path[->,>=stealth, dashed] (s9) edge (s10);
				\path[->,>=stealth, dashed] (s10) edge (p4);
				
				\tikzset{mystyle/.style={->,relative=false,in=0,out=0}}				
				\draw [->,>=stealth, red, line width=0.6] (p4) to [bend right] (s2);
				\draw [->,>=stealth, red, line width=0.6] (p4) to [bend left] (s10);
				\draw [->,>=stealth, red, line width=0.6] (p4) to [bend left] (s9);
				\draw [->,>=stealth, red, line width=0.6] (p4) to (s7);		
				\draw [->,>=stealth, red, line width=0.6] (p4) to (s1);
				
				\draw [->,>=stealth, dashed] (p4) to [bend left] ($(s9)+(.5,-1)$) to
				($(p2)+(0,-1)$) to [bend left] (s3);
				\draw [->,>=stealth, dashed] (p4) to (s4);
				\draw [->,>=stealth, dashed] (p4) to [bend right] (s6);
				\draw [->,>=stealth, dashed] (p4) to (s5);
				\draw [->,>=stealth, dashed] (p4) to [bend right] (s8);
			\end{tikzpicture}
			\caption{{The PCMCA-WT instance associated with the RSA instance depicted in Figure \ref{rsa-example}. A RSA solution of minimum cost is given by the blue arcs. The red arcs have cost 0 and, together with the blue ones, form an optimal PCMCA-WT solution.}}
			\label{fig-rsa-pcmca}
		\end{figure*}
		
		If the instance of RSA has a solution of cost $k$, then a solution of cost $k$ for the instance of PCMCA-WT can be obtained. Starting from the solution of the RSA problem,  it is possible to complete the solution of the associated PCMCA-WT problem by adding 0-cost arcs (red arcs) to connect the node  $P_{FAR}$ to the Steiner nodes not used in the RSA solution. {The solution of an RSA instance and a solution of the associated PCMCA-WT problem are depicted in Figure \ref{fig-rsa-pcmca}.}
		
		{Conversely, assume that there is a feasible solution of PCMCA-WT with cost at most $k$. Without loss of generality suppose that such a solution is optimal. Note that a path starting at $P_{FAR}$ and passing through a vertex in $P$ cannot exist due to the precedence constraints. Besides, every leaf of the arborescence that is in $S$ must have $P_{FAR}$ as parent; otherwise, making $P_{FAR}$ its parent would reduce the cost. Therefore, removing all the leaves of the PCMCA-WT arborescence connected through $P_{FAR}$ results in a tree that uses only arcs in $A'$ and whose leaves are all in $P$. It follows that the resulting tree is a feasible solution for the RSA.}

	\end{sloppypar}
\end{proof}

\subsection{MILP Models}
\label{sec-pcmrt-models}

This section introduces three different MILP models for formulating the \textit{Precedence-Constrained Minimum-Cost Arborescence Problem with Waiting Times}. For all the models, let $d_j$ be a variable associated with every vertex $j \in V$ to represent the time at which the flow {enters} vertex $j$, {with $d_r = 0$}. {The value of $d_j$ is bounded {from below} by summing the {time} from $r$ to the parent $i$ of $j$ and the cost of the arc $(i, j) \in A$ that is part of the arborescence. To ensure that the resulting arborescence satisfies the precedence constraints, we enforce that the {time}  from $r$ to $t$ is greater than or equal to the {time}  from $r$ to $s$ for all $(s,t) \in R$.} A variable $x_{ij}$ is associated with every arc $(i,j) \in A$ such that $x_{ij} = 1$ if $(i,j) \in T$ and 0 otherwise, where $T$ is the resulting optimal arborescence.

In all the models proposed for the PCMCA-WT, the value of $M$, which is an upper bound on the value of {an} optimal solution, is equal to the solution cost of solving the instance as a Sequential Ordering Problem (SOP) \cite{ref-montemanni2} using a nearest neighbor algorithm \cite{ref-shobaki}. This is a valid upper bound on the solution for the PCMCA-WT, since a valid solution for the SOP consists of a simple directed path that includes all the vertices of the graph such that $t$ never precede $s$ for all $(s,t) \in R$, which implies that $d_t \geq d_s$ for all $(s,t) \in R$, and the waiting time on each vertex is equal to zero.

\subsubsection{A Multi-Commodity Flow Model}
\label{sec-pcmrt-NF-models}

The model introduced in this section extends the one introduced in \cite{ref-dellamico} for the PCMCA, {and formulates the sub-problem of finding an arborescence rooted at $r$ that does not violate precedence relationships in $R$ as a \textit{multi-commodity flow problem}. The model uses a polynomial set of constraints instead of inequalities (\ref{eqn-3}) to ensure that every vertex in the graph is reachable from the root, and that for any $(s,t) \in R$ there is no path from $r$ to $s$ that passes through $t$ in the resulting arborescence. This can be ensured by having a flow value of 1 that enters every vertex $k$ in the graph, and that for any vertex $k$ the flow to that vertex does not pass through a successor of $k$.} Let $y^{k}_{ij}$ be a variable associated with every vertex $k \in V \backslash \{r\}$ and every arc $(i,j) \in A$, such that $y^{k}_{ij} = 1$ if arc $(i,j) \in A$ is part of the path from the root $r$ to vertex $k$, and 0 otherwise. Let $w_i$ be the waiting time at vertex $i \in V$.

\allowdisplaybreaks
\begin{align}
	\text{(MCF)} \,\, \text{minimize} \mkern100mu & \mkern-100mu \sum_{\substack{(i,j) \in A}} c_{ij} x_{ij} + \sum_{i \in V} w_i
	\label{eqn-5}
	\\
	\text{subject to} \mkern150mu & \mkern-150mu \sum_{(i,j) \in A} x_{ij} = 1 \mkern260mu \forall j \in V \backslash \{r\}
	\label{eqn-6}
	\\
	& \mkern-150mu \sum_{\substack{(i,j) \in A : \\ (k, j) \notin R}} y^{k}_{ij} - \sum_{\substack{(j,i) \in A : \\ (k, j) \notin R}} y^{k}_{ji} = \left\{
	\begin{array}{l}
		1 \mbox{ if } i = r \\
		\mkern-14mu -1 \mbox{ if } i = k \\
		 0 \mbox{ otherwise}
	\end{array} \right.
	\mkern-15mu \begin{array}{r} 
		\forall k \in V \backslash \{r\}, \\
		\forall i \in V: (k,i) \notin R
	\end{array}
	\label{eqn-7}
	\\
	& \mkern-150mu d_r = 0 
	\label{eqn-8}
	\\
	& \mkern-150mu {d_j \geq d_i - M + (M + c_{ij})x_{ij} \mkern160mu \forall (i,j) \in A}
	\label{eqn-9}
	\\
	& \mkern-150mu w_j \geq d_j - d_i - M + (M - c_{ij})x_{ij} \mkern118mu \forall (i,j) \in A
	\label{eqn-10}
	\\
	& \mkern-150mu d_t \geq d_s \mkern320mu \forall (s, t) \in R
	\label{eqn-11}
	\\
	& \mkern-150mu y^{k}_{ij} \leq x_{ij} \mkern220mu \forall k \in V \backslash \{r\}, (i,j) \in A
	\label{eqn-12}
	\\
	& \mkern-150mu x_{ij} \in \{0,1\} \mkern288mu \forall (i,j) \in A
	\label{eqn-13}
	\\
	& \mkern-150mu y^{k}_{ij} \in \{0, 1\} \mkern200mu \forall k \in V \backslash \{r\}, (i,j) \in A
	\label{eqn-14}
	\\
	& \mkern-150mu d_i, w_i \geq 0 \mkern332mu \forall i \in V
	\label{eqn-15}
\end{align}

Constraints (\ref{eqn-6}) impose the first property of an arborescence namely that every vertex $v \in V \backslash \{r\}$ must have a single parent.
{Constraints (\ref{eqn-7}) are the multi-commodity flow constraints: every vertex $k \in V$ must be reachable from the root, and any path from $r$ to $k$ must not pass through the successors of $k$ in the precedence graph $P$  (otherwise this would violate a precedence relation).} Constraint (\ref{eqn-8}) sets the distance from the root $r$ to itself to be equal to 0. {Constraints (\ref{eqn-9}) impose that when arc $(i,j)$ is selected to be part of the arborescence, then the time at which the flow enters vertex $j$ is greater than or equal to the time at which the flow enters vertex $i$ plus $c_{ij}$. Constraints (\ref{eqn-10}) enforce that the waiting time at {each} vertex $j$ is greater than or equal to the difference between the time at which the flow {enters} vertex $j$ and the time at which the flow {enters} vertex $i$ plus $c_{ij}$, where $i$ is the parent of $j$ in the arborescence. Constraints (\ref{eqn-11}) enforce that the time at which the flow {enters} vertex $t$ must be greater than or equal to the time at which the flow {enters} vertex $s$, for all $(s,t) \in R$. Finally, constraints (\ref{eqn-12})-(\ref{eqn-15}) define the domain of the variables. The MILP model proposed has $O(|V||A|)$ variables, and $O(|V||A|)$ constraints.

The major drawback of this model is the large number of variables used which might result in memory issues when solving large-sized instances, similar to what happens in the model proposed in \cite{ref-dellamico} for the PCMCA.

\subsubsection{A Distance-Accumulation Model}
\label{sec-pcmrt-DA-models}

The model introduced in this section extends the model introduced in Section \ref{sec-pcmca-model} for the PCMCA. {As mentioned earlier,} the {time}  from the root $r$ to vertex $j$ in the arborescence is bounded from below  by summing the {time}  from $r$ to the parent $i$ of $j$ and the cost of the arc $(i, j) \in A$, with $d_r = 0$. To ensure that the resulting arborescence satisfies the precedence constraints, we enforce that the {time}  from $r$ to $t$ is greater than or equal to the {time}  from $r$ to $s$ for all $(s,t) \in R$. {We recall that} $w_i$ is the waiting time at vertex $i \in V$.

\begin{align}
	\text{(DA)} \,\, \text{minimize} & \sum_{\substack{(i,j) \in A}} c_{ij} x_{ij} + \sum_{i \in V} w_i 
	\label{eqn-16}
	\\
	\text{subject to} & \sum_{(i,j) \in A} x_{ij} = 1 & \forall j \in V \backslash \{r\}
	\label{eqn-17}
	\\
	& \sum_{\substack{(i,k) \in A:\\ i \in V_j \backslash S, \, k \in S}} x_{ik} \geq 1 & \mkern-125mu \forall j \in V \setminus \{r\}, \forall S \subseteq V_j \backslash \{r\}: j \in S
	\label{eqn-18}
	\\
	& d_r = 0  
	\label{eqn-19}
	\\
	& d_j \geq d_i - M + (M + c_{ij})x_{ij} & \forall (i,j) \in A
	\label{eqn-20}
	\\
	& w_j \geq d_j - d_i - M + (M - c_{ij})x_{ij} & \forall (i,j) \in A
	\label{eqn-21}
	\\
	& d_t \geq d_s & \forall (s, t) \in R
	\label{eqn-22}
	\\
	& x_{ij} \in \{0,1\} & \forall (i,j) \in A
	\label{eqn-23}
	\\
	& d_i, w_i \geq 0 & \forall i \in V
	\label{eqn-24}
\end{align}

Constraints (\ref{eqn-17}) impose the first property of an arborescence, namely that every vertex $v \in V \backslash \{r\}$ must have a single parent. Constraints (\ref{eqn-18}) model the connectivity constraint, that is every vertex $v \in V \backslash \{r\}$ must be reachable from the root, and they also impose the precedence constraints where the resulting arborescence should not include vertex $t$ in the directed path connecting $r$ to $s$ when $(s,t) \in R$. This will lead to an arborescence such that the flow never {enters} $t$ before {entering} $s$, if $s$ precedes $t$ on the same directed path. Constraint (\ref{eqn-19}) sets the distance from the root $r$ to itself to be equal to 0. Constraints (\ref{eqn-20}) impose that when arc $(i,j)$ is selected to be part of the arborescence, then the time at which the flow {enters} vertex $j$ is greater than or equal to the time at which the flow {enters} vertex $i$ plus $c_{ij}$. Constraints (\ref{eqn-21}) enforce that the waiting time at vertex $j$ is greater than or equal to the difference between the time at which the flow {enters} vertex $j$ and the time at which the flow {enters} vertex $i$ plus $c_{ij}$. Constraints (\ref{eqn-22}) enforce that the time at which the flow {enters} vertex $t$ is greater than or equal to the time at which the flow {enters} vertex $s$ for all $(s,t) \in R$. Finally, constraints (\ref{eqn-23}) and (\ref{eqn-24}) define the domain of the variables. The MILP model proposed, {without constraints (\ref{eqn-18}),} has $O(|A|)$ variables, and $O(|A|)$ constraints. Constraints (\ref{eqn-18}) are dynamically added to the model using the same separation procedure described in Section \ref{sec-pcmca-model}.

\subsubsection{An Adjusted Arc-Cost Model}
\label{sec-alt-model}

{The model introduced in this section is originated by removing inequalities (\ref{eqn-21}) from the model introduced in Section \ref{sec-pcmrt-DA-models} and representing the value of $w_j$ by the nonlinear term}
\begin{align}
& {w_j = \sum_{i:(i,j) \in A} (d_j- d_i - c_{ij})x_{ij}}
\label{eqn-wj}
\end{align}

A different linear model is then derived as follows.

\begin{prop}
	\label{prop-2}
	The waiting time at vertex $j \in V$ can be expressed by the nonlinear {equality} (\ref{eqn-wj}).
\end{prop}

\begin{proof}
	Inequalities (\ref{eqn-21}) can be rewritten as $w_j \geq d_j - d_i - c_{ij} - M(1 - x_{ij})$ $\forall (i,j) \in A$. If $x_{ij} = 0$ then $w_j$ has to be greater than or equal to a negative value, however the value of $w_j$ should be greater than or equal to zero by definition. Accordingly, the inequality would be active and affect the solution only when $x_{ij} = 1$. {Therefore, we can represent the waiting time at vertex $j$ using equality (\ref{eqn-wj}).} 
\end{proof}

{Based on Proposition \ref{prop-2}, we can replace the second term in the objective function (\ref{eqn-16}) as follows:}
\begin{align*}
	{\sum_{j\in V} w_j = \sum_{j \in V} \sum_{i:(i,j)\in A} (d_j - d_i - c_{ij})x_{ij} = \sum_{(i,j)\in A} (d_j - d_i - c_{ij})x_{ij}}
\end{align*}
{This means that inequalities (\ref{eqn-21}) are no longer necessary as the objective function no longer depends on $w$, which results in the following nonlinear model.}

\allowdisplaybreaks
\begin{align}
	\text{minimize} & \sum_{(i,j) \in A} c_{ij} x_{ij} + \sum_{(i,j) \in A} (d_j - d_i - c_{ij})x_{ij}
	\label{eqn-25}
	\\
	\text{subject to} & \sum_{(i,j) \in A} x_{ij} = 1 & \forall j \in V \backslash \{r\}
	\label{eqn-26}
	\\
	& \sum_{\substack{(i,k) \in A:\\ i \in V_j \backslash S, \, k \in S}} x_{ik} \geq 1 & \mkern-110mu \forall j \in V \setminus \{r\}, \forall S \subseteq V_j \backslash \{r\}: j \in S
	\label{eqn-27}
	\\
	& d_r = 0
	\label{eqn-28}
	\\
	& d_j \geq d_i - M + (M + c_{ij})x_{ij} & \forall (i,j) \in A
	\label{eqn-29}
	\\
	& d_t \geq d_s & \forall (s, t) \in R
	\label{eqn-30}
	\\
	& x_{ij} \in \{0,1\} & \forall (i,j) \in A
	\label{eqn-31}
	\\
	& d_i \geq 0 & \forall i \in V
	\label{eqn-32}
\end{align}

\begin{prop}
	Using a new set of $|A|$ variables $z$ and $2|A|$ new constraints, the objective function (\ref{eqn-25}) can be linearized as follows:
	\begin{align*}
		\text{minimize} \, \sum_{j \in V \backslash \{r\}} d_j - \sum_{(i,j) \in A} z_{ij}
	\end{align*}
	\label{prop-3}
\end{prop}
\begin{proof}
	The objective function (\ref{eqn-25}) can be rewritten as follows:
	
	\begin{equation}
	\begin{aligned}
		& \sum_{(i,j) \in A} c_{ij} x_{ij} + \sum_{(i,j) \in A} (d_j - d_i - c_{ij})x_{ij} =
		\\
		& \sum_{(i,j) \in A} d_j x_{ij} - \sum_{(i,j) \in A} d_i x_{ij}
		\,\, = \sum_{j \in V \backslash \{r\}} d_j - \sum_{(i,j) \in A} d_i x_{ij}
	\end{aligned}
	\label{eqn-prop}
	\end{equation}
	We use the fact that $\sum_{(i,j) \in A} d_j x_{ij} = \sum_{j \in V \backslash \{r\}} d_j$ as each $j \in V \backslash \{r\}$ has exactly one $x_{ij}$ assigned to 1 in an arborescence, as imposed by (\ref{eqn-26}). 
	
	\noindent
	Since the term $d_i x_{ij}$ is summed over each arc $(i,j) \in A$, then we need at least $2|A|$ constraints to linearize the product. We can substitute each term $d_i x_{ij}$ by a new continuous variable $z_{ij}$ and the following two inequalities:  
	\begin{align}
	&z_{ij} \leq M x_{ij} & \forall (i,j) \in A  \label{eqn-n2}\\
	 &z_{ij} \leq d_i & \forall (i,j) \in A \label{eqn-n1}
	\end{align}
	Inequalities (\ref{eqn-n2}) ensure that if $x_{ij} = 0$ then $z_{ij} = 0$. On the other hand, if $x_{ij} = 1$, then inequalities (\ref{eqn-n2}) ensure that $z_{ij}$ is less than the upper bound on the optimal solution which is further tightened by inequalities (\ref{eqn-n1}). This results in a total of $2|A|$ new constraints and (\ref{eqn-25}) can now be expressed as $\sum_{j \in V \backslash \{r\}} d_j - \sum_{(i,j) \in A} z_{ij}$ by elaborating on (\ref{eqn-prop}).  
\end{proof}
	
	Based on Proposition \ref{prop-3}, we can derive the following MILP model that contains $O(|A|)$ variables, and $O(|A|)$ constraints, plus an exponential number of constraints (\ref{eqn-27}).

\begin{align}
	\text{(AAC)} \,\, \text{minimize} & \sum_{j \in V \backslash \{r\}} d_j - \sum_{(i,j) \in A} z_{ij} 
	\label{eqn-33}
	\\
	\text{subject to} & \sum_{(i,j) \in A} x_{ij} = 1 & \forall j \in V \backslash \{r\}
	\label{eqn-34}
	\\
	& \sum_{\substack{(i,k) \in A:\\ i \in V_j \backslash S, \, k \in S}} x_{ik} \geq 1 & \mkern-90mu \forall j \in V \setminus \{r\}, \forall S \subseteq V_j \backslash \{r\}: j \in S
	\label{eqn-35}
	\\
	& d_r = 0
	\label{eqn-36}
	\\
	& d_j \geq d_i - M + (M + c_{ij})x_{ij} & \forall (i,j) \in A
	\label{eqn-37}
	\\
	& d_t \geq d_s & \forall (s, t) \in R
	\label{eqn-38}
	\\
	& z_{ij} \leq d_i & \forall (i,j) \in A
	\label{eqn-39}
	\\
	& z_{ij} \leq M x_{ij} & \forall (i,j) \in A 
	\label{eqn-40}
	\\
	& x_{ij} \in \{0,1\} & \forall (i,j) \in A
	\label{eqn-41}
	\\
	& z_{ij} \geq 0 & \forall (i,j) \in A
	\label{eqn-42}
	\\
	& d_i \geq 0 & \forall i \in V
	\label{eqn-43}
\end{align}

\begin{prop}
	\label{prop-4}
	The following inequalities are valid for the (AAC) model:
	\begin{align}
		\sum_{i \in V: (i,j) \in A} z_{ij} \leq d_j - \sum_{(i,j) \in A} c_{ij}x_{ij} & \qquad \forall j \in V \backslash \{r\}
		\label{eqn-44}
	\end{align}
\end{prop}
\begin{proof}
	Since for each vertex $j \in V \backslash \{r\}$ there is only one active arc $(i, j) \in A$ entering $j$ (from inequalities (\ref{eqn-34})), from inequalities (\ref{eqn-37}) we can derive the following new quadratic inequalities:
	\begin{align}
		d_j \geq \sum_{i \in V: (i,j) \in A} d_i x_{ij} + \sum_{(i,j) \in A} c_{ij} x_{ij} & \qquad \forall j \in V \backslash \{r\}
		\label{eqn-45}
	\end{align}
	From inequalities (\ref{eqn-39}) and (\ref{eqn-40}) we have $z_{ij} \leq d_i x_{ij}$ (see Proposition \ref{prop-3}), then inequality (\ref{eqn-44}) can be derived from inequality (\ref{eqn-45}) as follows.
	\begin{align*}
		d_j \geq \sum_{i \in V: (i,j) \in A} d_i x_{ij} + \sum_{(i,j) \in A} c_{ij} x_{ij} & \implies
		\sum_{i \in V: (i,j) \in A} d_i x_{ij} \leq d_j - \sum_{(i,j) \in A} c_{ij} x_{ij} 
		\\ 
		& \implies \sum_{i \in V: (i,j) \in A} z_{ij} \leq d_j - \sum_{(i,j) \in A} c_{ij}x_{ij}
		\\
		& \implies d_j \geq \sum_{i \in V: (i,j) \in A} z_{ij} + \sum_{(i,j) \in A} c_{ij}x_{ij}
	\end{align*} 
\end{proof}

It should be noted that inequalities (\ref{eqn-44}) are not an integral part of the $AAC$ model, but are added to have a stronger linear relaxation. If the inequalities are not included in the model, then the value of the $z_{ij}$s can be substantially larger than the value of the $d_j$s in order to minimize the value of the objective function. {This could result in feasible solutions of the linear relaxation with a negative objective function}. This would make the MILP much harder to solve. {Therefore}, inequalities (\ref{eqn-44}) are considered for all the experiments reported in this paper.

\section{Experimental Results}
\label{sec-results}

The computational experiments for evaluating the proposed models are based on the benchmark instances of TSPLIB \cite{ref-reinelt}, SOPLIB \cite{ref-montemanni} and COMPILERS \cite{ref-shobaki} originally proposed for the SOP \cite{ref-SOP}. The benchmark instances are the same instances previously adopted in \cite{ref-dellamico}.

All the experiments are performed on an Intel i7 processor running at 1.8 GHz with 8 GB of RAM. CPLEX 12.8\footnote{IBM ILOG CPLEX Optimization Studio: \url{https://www.ibm.com/products/ilog-cplex-optimization-studio}} is used for solving the MILPs. CPLEX is run with its default parameters, and single threaded standard Branch-and-Cut (B\&C) algorithm is applied for solving the MILP models, with \emph{BestBound} node selection, and MIP emphasis set to \emph{MIPEmphasisOptimality}. A time limit of 1 hour is {set for} the computation time for each computational (new) method/instance. {No time limit was instead considered for the computational time of the \textit{Path-Based Model} (see \cite{ref-dellamico})}.

In all the tables that follow, \textit{Name} and \textit{Size} columns report the name and size of the instance, \textit{Density of P} reports the density of arcs in the precedence graph computed as \mbox{\large$\frac{2 \cdot |R|}{|V|(|V| - 1)}$}, {$z^*$ reports the value of the optimal solution for that instance.} For each model we report the following columns. \textit{Cuts} column reports the number of model-dependent cuts (inequalities) that are dynamically added to the model, \textit{Nodes} column reports the number of nodes in the search decision-tree, \textit{Time (s)} reports the solution time in seconds. The same set of columns is reported for both the results of the model's linear relaxation (grouped under $LR$), and for the mixed integer linear programming model (grouped under $IP$).
 
\subsection{Computational Results for the PCMCA}

A MILP model for the PCMCA was previously proposed in \cite{ref-dellamico}, where precedence constraints are imposed by propagating a value along every path with end-points $s$ and $t$ for $(s,t) \in R$ in order to detect a precedence violation. This results in a cubic number of variables (a variable for each precedence relationship and vertex), and a quadratic number of constraints for the value propagation. The model is known to suffer from scalability and performance issues \cite{ref-dellamico}. Tables \ref{table-comparison-tsplib}-\ref{table-comparison-compilers} report the overall results of the model proposed in Section \ref{sec-pcmca}, that will be named \textit{Set-Based Model}, and compare its results with the results obtained by the model previously proposed in \cite{ref-dellamico}, that is here named \textit{Path-Based Model}. {In Tables \ref{table-comparison-tsplib}-\ref{table-comparison-compilers}, the \textit{Gap} column indicates the percentage relative difference between the optimal solution {($z^*$)} of the PCMCA instance and the objective function value of the model's linear relaxation ($Cost_{LR}$), computed as {$100 \cdot \frac{z^* - Cost_{LR}}{z^*}$.}

An overview of the results for the \textit{Path-Based Model} shows that its linear relaxation optimally solves 47\% of the instances with a 2.1\% average optimality gap. On the other hand, the linear relaxation of the \textit{Set-Based Model} optimally solves 68\% of the instances (a 44\% improvement compared to \textit{Path-Based Model}) with an average optimality gap of 1.7\% (a 23\% improvement compared to the \textit{Path-Based Model}). The solution times for the integer \textit{Path-Based Model} range between milliseconds and 2.5 hours (the maximum computing time allowed was longer in \cite{ref-dellamico}), with an average of 276 seconds, a median of 3 seconds, and standard deviation of 1116 seconds. The solution times for the integer \textit{Set-Based Model} range between milliseconds and 15 minutes, with an average of 27 seconds, a median of 0.8 seconds, and standard deviation of 129 seconds (this is on average a 90\% improvement compared to the \textit{Path-Based Model}). In the integer \textit{Set-Based Model}, the number of cuts generated by exploring the whole branch-decision-tree increases by 80\% on average, compared to the root of the branch-decision-tree itself, and the solver explores 77 nodes on average. {On the other hand, for the integer \textit{Path-Based Model} the solver explores 5588 nodes on average (a 98\% increase).}

By inspecting Tables \ref{table-comparison-tsplib}-\ref{table-comparison-compilers} we can observe that the \textit{Path-Based Model} from \cite{ref-dellamico} optimally solves a subset of the instances faster than the \textit{Set-Based Model}. We can see that those instances (underlined in the tables) are relatively large in size and have either a very sparse or very dense precedence graph. More specifically, in terms of size the \textit{Path-Based Model} is faster at solving 54\% of the instances that have a size larger than 500. In terms of precedence graph density, the \textit{Path-Based Model} is faster at solving 62\% of the instance with density smaller than 0.005 and is faster at solving 57\% of the instances with density larger than 0.990. Considering the two factors simultaneously, the \textit{Path-Based Model} is faster at solving 57\% of the instances with size larger than 500 and precedence graph density that is smaller than 0.008 or larger than 0.940. A low density precedence graph implies a small number of variables and constraints used to model the precedence relationships in the \textit{Path-Based Model}, and since finding a violated precedence inequality is much faster in that model, it is sometimes more efficient at solving those instances. In other instances, the increase in solution time is justified by the time it takes to find a violated inequality in the \textit{Set-Based Model}. In general, if we look at Figure \ref{Fig-P1-Solutiontimes}, which shows the distribution of solution times for each model, we see that the \textit{Set-Based Model} is much faster at solving the instances, even when we consider or exclude outliers. The large solution time in the \textit{Set-Based Model} for the two instances \emph{prob.100} and \emph{R.700.100.1}, compared to the \textit{Path-Based Model}, can be explained by the number of cuts generated while solving the LR, which also increases the overall solution time. We can verify that by observing the solution time of the LR for the first instance.

\begin{figure}[]
	\resizebox{\textwidth}{!}{
		\includegraphics{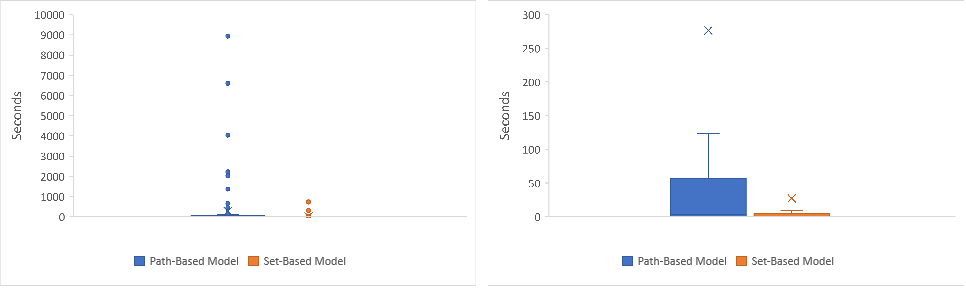}
	}
	\centering
	\caption{A box plot showing the distribution of solution times (in seconds) of all the 116 instances for the \textit{Path-Based Model} from \cite{ref-dellamico} and the \textit{Set-Based Model}. The box plot on the right excludes outliers.}
	\label{Fig-P1-Solutiontimes}
\end{figure}

In conclusion, the \textit{Set-Based Model} is a significant improvement over previous methods. Indeed, it provides optimal solutions in substantially less time, and memory usage for the majority of the instances considered. The same cannot be said about the \textit{Set-Based Model} when linear relaxations only are considered, as the \textit{Path-Based Model} can be solved much faster in most cases because of the fewer number of constraints, although it sometimes generates a looser estimate on the value of the optimal integer solution. In terms of memory usage, the \textit{Set-Based Model} consumes approximately an average of 95MB, with a standard deviation of 132 and median of 41 when solving the instances considered. On the other hand, the \textit{Path-Based Model} consumes approximately an average of 363MB, with a standard deviation of 486 and median of 116. For considerably large sized instances such as \emph{R.700.100.30} and \emph{R.700.100.60}, the \textit{Path-Based Model} consumes 1841MB and 1097MB for each of those two instances, whereas the \textit{Set-Based Model} consumes 263MB and 286MB for the same instances. The two instances considered are solved at the root node of the branch-decision tree {by both models}.

\subsection{Computational Results for the PCMCA-WT}

For the computational experiments of the PCMCA-WT, we omit the {detailed} results for SOPLIB benchmark instances. {However, we can draw the following conclusions on these instances}. The \textit{Multi-Commodity Flow} (MCF) model is unable to solve large sized instances because of memory issues (building the model consumes around 5GB of memory on average) or time out while solving the model's linear relaxation. Since the linear relaxation of MCF model was unable to solve {any} single instance from SOPLIB benchmark set, we concluded that it is highly unsuitable for solving such instances. The characteristics of SOPLIB instances are summarized in Table \ref{table-comparison-soplib}. The \textit{Distance-Accumulation} (DA) model and the \textit{Adjusted Arc-Cost} (AAC) model are able to optimally solve SOPLIB instances with low density precedence graphs within the time limit, with an average of 800 seconds, and achieve an average optimality gap of 63.8\% for the LR models and 63.3\% for the IP models before timing out for the remaining instances. These figures are much higher compared to the other two benchmark sets as will be shown later. The computational experiments have shown that large-sized instances with a highly dense precedence graph are outside the reach of the models proposed due to the intrinsic complexity of the problem.

\begin{landscape}

\begin{table}[t]
	\caption{Overall computational results for comparing the \textit{Path-Based Model} from \cite{ref-dellamico} and the \textit{Set-Based Model} for the PCMCA for TSPLIB instances}
	\label{table-comparison-tsplib}
	\resizebox{1\linewidth}{0.45\textwidth}{
		\begin{tabular}{lrrrrrrrrrrrrr}
			\hline
			\multicolumn{4}{c}{\multirow{2}{*}{Instance}}          & \multicolumn{4}{c}{Path-Based Model \cite{ref-dellamico}}                        & \multicolumn{6}{c}{Set-Based Model}                                 \\
			\cmidrule(r){5-8} \cmidrule(r){9-14}
			\multicolumn{4}{c}{}                                   & \multicolumn{2}{c}{LR}             & \multicolumn{2}{c}{IP} & \multicolumn{3}{c}{LR}                    & \multicolumn{3}{c}{IP} \\
			\cmidrule(r){1-4} \cmidrule(r){5-6} \cmidrule(r){7-8} \cmidrule(r){9-11} \cmidrule(r){12-14}
			\multicolumn{1}{l}{Name} & Size & Density of P & $z^*$ & \multicolumn{1}{l}{Time (s)} & Gap & Nodes    & Time (s)    & \multicolumn{1}{l}{Cuts} & Time (s) & Gap & Nodes & Cuts & Time (s) \\
			\hline
			br17.10 & 18 & 0.314 & 25 & 0.032 & 0.000 & 3 & 0.060 & 21 & 0.015 & 0.000 & 0 & 21 & 0.015 \\
			br17.12 & 18 & 0.359 & 25 & 0.047 & 0.000 & 3 & 0.063 & 22 & 0.016 & 0.000 & 0 & 22 & 0.016 \\
			ESC07 & 9 & 0.611 & 1531 & 0.031 & 0.000 & 0 & 0.031 & 13 & 0.031 & 0.000 & 0 & 13 & 0.031 \\
			ESC11 & 13 & 0.359 & 1752 & 0.031 & 0.000 & 0 & 0.031 & 1 & 0.031 & 0.000 & 0 & 1 & 0.031 \\
			ESC12 & 14 & 0.396 & 1138 & 0.016 & 0.000 & 0 & 0.016 & 1 & 0.016 & 0.000 & 0 & 1 & 0.016 \\
			ESC25 & 27 & 0.177 & 1041 & 0.062 & 0.000 & 0 & 0.062 & 31 & 0.063 & 0.000 & 0 & 31 & 0.063 \\
			ESC47 & 49 & 0.108 & 703 & 0.484 & 0.284 & 5 & 0.469 & 142 & 0.547 & 0.000 & 0 & 142 & 0.547 \\
			ESC63 & 65 & 0.173 & 56 & 0.329 & 0.000 & 0 & 0.329 & 42 & 0.218 & 0.000 & 0 & 42 & 0.218 \\
			ESC78 & 80 & 0.139 & 502 & 0.094 & 0.000 & 0 & 0.094 & 1 & 0.047 & 0.000 & 0 & 1 & 0.047 \\
			ft53.1 & 54 & 0.082 & 3917 & 1.172 & 0.408 & 7 & 1.172 & 78 & 0.328 & 0.230 & 5 & 84 & 0.375 \\
			ft53.2 & 54 & 0.094 & 3978 & 0.281 & 7.642 & 104 & 0.688 & 57 & 0.188 & 2.765 & 55 & 211 & 0.547 \\
			ft53.3 & 54 & 0.225 & 4242 & 1.890 & 5.587 & 122 & 2.547 & 96 & 0.453 & 0.000 & 0 & 96 & 0.453 \\
			ft53.4 & 54 & 0.604 & 4882 & 0.156 & 2.663 & 9 & 0.250 & 13 & 0.047 & 0.000 & 0 & 13 & 0.047 \\
			ft70.1 & 71 & 0.036 & 32846 & 2.891 & 0.000 & 1 & 2.828 & 158 & 2.750 & 0.000 & 0 & 158 & 2.750 \\
			ft70.2 & 71 & 0.075 & 32930 & 2.985 & 0.035 & 2 & 3.016 & 163 & 2.719 & 0.000 & 0 & 163 & 2.719 \\
			ft70.3 & 71 & 0.142 & 33431 & 0.750 & 2.423 & 954 & 63.171 & 66 & 0.265 & 2.034 & 145 & 2077 & 38.250 \\
			ft70.4 & 71 & 0.589 & 35179 & 13.015 & 0.584 & 53 & 13.438 & 30 & 0.094 & 2.146 & 369 & 1070 & 6.281 \\
			rbg048a & 50 & 0.444 & 204 & 0.047 & 0.000 & 0 & 0.047 & 5 & 0.031 & 0.000 & 0 & 5 & 0.031 \\
			rbg050c & 52 & 0.459 & 191 & 0.313 & 0.000 & 0 & 0.313 & 11 & 0.047 & 0.000 & 0 & 11 & 0.047 \\
			rbg109 & 111 & 0.909 & 256 & 11.578 & 0.000 & 0 & 11.578 & 14 & 0.094 & 0.000 & 0 & 14 & 0.094 \\
			rbg150a & 152 & 0.927 & 373 & 2.485 & 0.000 & 0 & 2.485 & 14 & 0.187 & 0.000 & 1 & 14 & 0.219 \\
			rbg174a & 176 & 0.929 & 365 & 29.610 & 0.274 & 2 & 29.609 & 22 & 0.297 & 0.000 & 1 & 22 & 0.313 \\
			rbg253a & 255 & 0.948 & 375 & 13.985 & 0.000 & 0 & 13.985 & 22 & 1.125 & 0.000 & 0 & 22 & 1.125 \\
			rbg323a & 325 & 0.928 & 754 & 1.547 & 0.000 & 0 & 1.547 & 26 & 1.047 & 0.000 & 0 & 26 & 1.047 \\
			rbg341a & 343 & 0.937 & 610 & 23.344 & 3.279 & 376 & 278.859 & 89 & 3.031 & 0.000 & 0 & 89 & 3.031 \\
			rbg358a & 360 & 0.886 & 595 & 0.312 & 0.000 & 0 & \underline{0.312} & 67 & 5.812 & 0.000 & 0 & 67 & 5.812 \\
			rbg378a & 380 & 0.894 & 559 & 16.079 & 3.936 & 543 & 178.515 & 21 & 1.829 & 4.472 & 36 & 282 & 19.047 \\
			kro124p.1 & 101 & 0.046 & 32597 & 0.734 & 5.997 & 47 & 1.844 & 95 & 1.782 & 0.000 & 0 & 95 & 1.782 \\
			kro124p.2 & 101 & 0.053 & 32851 & 0.578 & 6.929 & 1433 & 11.203 & 109 & 1.828 & 0.568 & 27 & 238 & 3.281 \\
			kro124p.3 & 101 & 0.092 & 33779 & 8.672 & 2.680 & 258648 & 6599.140 & 69 & 0.844 & 3.486 & 98 & 656 & 7.469 \\
			kro124p.4 & 101 & 0.496 & 37124 & 41.828 & 1.375 & 198 & 59.359 & 128 & 1.672 & 0.000 & 0 & 128 & 1.672 \\
			p43.1 & 44 & 0.101 & 2720 & 0.594 & 12.684 & 238 & 4.203 & 68 & 0.187 & 10.409 & 128 & 692 & 1.765 \\
			p43.2 & 44 & 0.126 & 2720 & 1.016 & 8.364 & 119 & \underline{1.781} & 33 & 0.079 & 11.029 & 237 & 1164 & 4.359 \\
			p43.3 & 44 & 0.191 & 2720 & 0.547 & 14.407 & 283 & 2.829 & 77 & 0.188 & 7.537 & 134 & 598 & 1.437 \\
			p43.4 & 44 & 0.164 & 2820 & 1.218 & 8.688 & 198 & 3.516 & 11 & 0.047 & 8.333 & 353 & 1065 & 2.797 \\
			prob.100 & 100 & 0.048 & 650 & 11.766 & 1.308 & 1428 & \underline{36.594} & 1840 & 622.437 & 0.240 & 4 & 1962 & 743.969 \\
			prob.42 & 42 & 0.116 & 143 & 0.125 & 0.000 & 0 & 0.125 & 2 & 0.032 & 0.000 & 0 & 2 & 0.032 \\
			ry48p.1 & 49 & 0.091 & 13095 & 0.828 & 0.886 & 879 & 1.656 & 31 & 0.094 & 1.894 & 54 & 177 & 0.609 \\
			ry48p.2 & 49 & 0.103 & 13103 & 1.031 & 0.551 & 220 & 1.593 & 58 & 0.235 & 0.000 & 0 & 58 & 0.235 \\
			ry48p.3 & 49 & 0.193 & 13886 & 2.109 & 3.657 & 123233 & 638.344 & 34 & 0.078 & 6.852 & 146 & 634 & 2.156 \\
			ry48p.4 & 49 & 0.588 & 15340 & 2.531 & 7.210 & 8610 & 24.156 & 65 & 0.172 & 5.847 & 32 & 153 & 0.313 \\
			Average &  &  &  & 4.808 & 2.484 & 9700 & 194.923 & 94 & 15.878 & 1.655 & 45 & 300 & 20.855 \\
			\hline
		\end{tabular}
	}
\end{table}

\begin{table}[t]
	\centering
	\caption{Overall computational results for comparing the \textit{Path-Based Model} from \cite{ref-dellamico} and \textit{Set-Based Model} for the PCMCA for SOPLIB instances}
	\label{table-comparison-soplib}
	\resizebox{1\linewidth}{0.45\textwidth}{
		\begin{tabular}{lrrrrrrrrrrrrr}
			\hline
			\multicolumn{4}{c}{\multirow{2}{*}{Instance}}          & \multicolumn{4}{c}{Path-Based Model \cite{ref-dellamico}}                        & \multicolumn{6}{c}{Set-Based Model}                                 \\
			\cmidrule(r){5-8} \cmidrule(r){9-14}
			\multicolumn{4}{c}{}                                   & \multicolumn{2}{c}{LR}             & \multicolumn{2}{c}{IP} & \multicolumn{3}{c}{LR}                    & \multicolumn{3}{c}{IP} \\
			\cmidrule(r){1-4} \cmidrule(r){5-6} \cmidrule(r){7-8} \cmidrule(r){9-11} \cmidrule(r){12-14}
			\multicolumn{1}{l}{Name} & Size & Density of P & $z^*$ & \multicolumn{1}{l}{Time (s)} & Gap & Nodes    & Time (s)    & \multicolumn{1}{l}{Cuts} & Time (s) & Gap & Nodes & Cuts & Time (s) \\
			\hline
			R.200.100.1 & 200 & 0.020 & 29 & 0.219 & 0.000 & 0 & 0.219 & 11 & 0.875 & 0.000 & 0 & 11 & 0.875 \\
			R.200.100.15 & 200 & 0.847 & 454 & 3235.391 & 5.740 & 382 & 4034.859 & 85 & 1.079 & 13.877 & 177 & 2395 & 64.812 \\
			R.200.100.30 & 200 & 0.957 & 529 & 12.922 & 11.153 & 59 & 54.828 & 39 & 0.266 & 9.263 & 10 & 77 & 0.875 \\
			R.200.100.60 & 200 & 0.991 & 6018 & 3.593 & 0.000 & 0 & 3.593 & 0 & 0.094 & 0.000 & 0 & 0 & 0.094 \\
			R.200.1000.1 & 200 & 0.020 & 887 & 0.203 & 0.000 & 0 & 0.203 & 3 & 0.656 & 0.000 & 0 & 3 & 0.656 \\
			R.200.1000.15 & 200 & 0.876 & 5891 & 203.234 & 4.261 & 132 & 329.313 & 35 & 0.766 & 5.568 & 87 & 557 & 7.860 \\
			R.200.1000.30 & 200 & 0.958 & 7653 & 56.000 & 0.026 & 2 & 57.141 & 9 & 0.234 & 0.000 & 0 & 9 & 0.297 \\
			R.200.1000.60 & 200 & 0.989 & 6666 & 3.797 & 0.000 & 0 & 3.797 & 0 & 0.094 & 0.000 & 0 & 0 & 0.094 \\
			R.300.100.1 & 300 & 0.013 & 13 & 0.500 & 0.000 & 0 & \underline{0.500} & 14 & 2.250 & 0.000 & 0 & 14 & 2.250 \\
			R.300.100.15 & 300 & 0.905 & 575 & 3.985 & 10.261 & 87859 & 2220.656 & 20 & 1.171 & 7.652 & 139 & 1111 & 55.734 \\
			R.300.100.30 & 300 & 0.970 & 756 & 1.672 & 0.000 & 0 & 1.672 & 27 & 0.562 & 0.000 & 0 & 27 & 0.562 \\
			R.300.100.60 & 300 & 0.994 & 708 & 1.531 & 0.000 & 2 & 2.469 & 2 & 0.297 & 0.000 & 0 & 2 & 0.375 \\
			R.300.1000.1 & 300 & 0.013 & 715 & 10.546 & 0.000 & 0 & 10.546 & 8 & 2.094 & 0.000 & 0 & 8 & 2.515 \\
			R.300.1000.15 & 300 & 0.905 & 6660 & 0.812 & 5.983 & 3304 & 91.938 & 136 & 2.610 & 0.811 & 73 & 819 & 16.531 \\
			R.300.1000.30 & 300 & 0.965 & 8693 & 1.531 & 0.000 & 0 & 1.531 & 6 & 0.391 & 0.000 & 0 & 6 & 0.453 \\
			R.300.1000.60 & 300 & 0.994 & 7678 & 23.234 & 0.000 & 0 & 23.234 & 2 & 0.297 & 0.000 & 0 & 2 & 0.297 \\
			R.400.100.1 & 400 & 0.010 & 6 & 0.391 & 0.000 & 0 & \underline{0.391} & 42 & 5.781 & 0.000 & 2 & 45 & 9.750 \\
			R.400.100.15 & 400 & 0.927 & 699 & 0.328 & 10.837 & 52858 & 2021.813 & 24 & 0.906 & 10.014 & 109 & 548 & 44.922 \\
			R.400.100.30 & 400 & 0.978 & 712 & 10.156 & 0.000 & 0 & 10.156 & 14 & 1.656 & 0.000 & 0 & 14 & 2.031 \\
			R.400.100.60 & 400 & 0.996 & 557 & 0.219 & 0.000 & 0 & 0.219 & 0 & 0.328 & 0.000 & 0 & 0 & 0.328 \\
			R.400.1000.1 & 400 & 0.010 & 780 & 6.734 & 0.000 & 0 & 6.734 & 4 & 2.797 & 0.000 & 0 & 4 & 2.797 \\
			R.400.1000.15 & 400 & 0.930 & 7382 & 0.625 & 8.467 & 56018 & 8935.188 & 78 & 5.375 & 2.181 & 91 & 362 & 24.000 \\
			R.400.1000.30 & 400 & 0.977 & 9368 & 34.531 & 1.057 & 4797 & 209.593 & 20 & 1.140 & 4.366 & 38 & 97 & 6.563 \\
			R.400.1000.60 & 400 & 0.995 & 7167 & 2.016 & 0.000 & 0 & 2.016 & 1 & 0.500 & 0.000 & 0 & 1 & 0.500 \\
			R.500.100.1 & 500 & 0.008 & 3 & 217.172 & 0.000 & 0 & 217.172 & 29 & 11.812 & 0.000 & 0 & 29 & 11.812 \\
			R.500.100.15 & 500 & 0.945 & 860 & 1.016 & 8.488 & 9879 & 443.125 & 100 & 7.406 & 3.895 & 38 & 286 & 21.156 \\
			R.500.100.30 & 500 & 0.980 & 710 & 14.453 & 3.099 & 11490 & 696.922 & 19 & 0.797 & 6.620 & 15 & 51 & 3.562 \\
			R.500.100.60 & 500 & 0.996 & 566 & 0.687 & 0.000 & 0 & \underline{0.687} & 1 & 0.844 & 0.000 & 0 & 1 & 0.844 \\
			R.500.1000.1 & 500 & 0.008 & 297 & 0.609 & 0.000 & 0 & \underline{0.609} & 0 & 4.469 & 0.000 & 0 & 0 & 4.469 \\
			R.500.1000.15 & 500 & 0.940 & 8063 & 82.015 & 0.000 & 57 & 100.640 & 119 & 15.063 & 0.000 & 0 & 119 & 15.063 \\
			R.500.1000.30 & 500 & 0.981 & 9409 & 11.141 & 0.000 & 0 & 11.141 & 11 & 3.125 & 0.000 & 0 & 11 & 3.125 \\
			R.500.1000.60 & 500 & 0.996 & 6163 & 0.671 & 0.000 & 0 & \underline{0.671} & 1 & 0.875 & 0.000 & 0 & 1 & 0.875 \\
			R.600.100.1 & 600 & 0.007 & 1 & 659.156 & 0.000 & 0 & \underline{659.156} & 1455 & 733.375 & 0.000 & 0 & 1455 & 733.375 \\
			R.600.100.15 & 600 & 0.950 & 568 & 31.516 & 0.000 & 1 & 34.985 & 23 & 5.312 & 0.000 & 0 & 23 & 5.312 \\
			R.600.100.30 & 600 & 0.985 & 776 & 13.484 & 1.675 & 659 & 298.109 & 24 & 2.375 & 0.000 & 0 & 24 & 2.375 \\
			R.600.100.60 & 600 & 0.997 & 538 & 0.359 & 0.000 & 0 & \underline{0.359} & 0 & 0.906 & 0.000 & 0 & 0 & 0.906 \\
			R.600.1000.1 & 600 & 0.007 & 322 & 0.844 & 0.000 & 0 & \underline{0.844} & 0 & 8.625 & 0.000 & 0 & 0 & 8.625 \\
			R.600.1000.15 & 600 & 0.945 & 9763 & 17.984 & 2.192 & 31 & 159.515 & 69 & 12.766 & 0.000 & 0 & 69 & 12.766 \\
			R.600.1000.30 & 600 & 0.984 & 9497 & 7.219 & 0.000 & 0 & 7.219 & 13 & 2.969 & 0.000 & 0 & 13 & 2.969 \\
			R.600.1000.60 & 600 & 0.997 & 6915 & 0.406 & 0.000 & 0 & \underline{0.406} & 0 & 0.922 & 0.000 & 0 & 0 & 0.922 \\
			R.700.100.1 & 700 & 0.006 & 2 & 1.250 & 0.000 & 0 & \underline{1.250} & 616 & 314.875 & 0.000 & 0 & 616 & 314.875 \\
			R.700.100.15 & 700 & 0.957 & 675 & 41.000 & 0.000 & 0 & 41.000 & 23 & 6.875 & 0.000 & 0 & 23 & 6.875 \\
			R.700.100.30 & 700 & 0.987 & 590 & 3.984 & 0.000 & 0 & 3.984 & 1 & 1.25 & 0.000 & 0 & 1 & 1.250 \\
			R.700.100.60 & 700 & 0.997 & 383 & 0.500 & 0.000 & 0 & 0.500 & 0 & 1.422 & 0.000 & 0 & 0 & 1.422 \\
			R.700.1000.1 & 700 & 0.006 & 611 & 1.625 & 0.000 & 0 & \underline{1.625} & 0 & 13.891 & 0.000 & 0 & 0 & 13.891 \\
			R.700.1000.15 & 700 & 0.956 & 2792 & 1.500 & 0.000 & 0 & \underline{1.500} & 4 & 1.875 & 0.000 & 0 & 4 & 1.875 \\
			R.700.1000.30 & 700 & 0.986 & 2658 & 0.360 & 0.000 & 0 & \underline{0.360} & 0 & 0.828 & 0.000 & 0 & 0 & 0.828 \\
			R.700.1000.60 & 700 & 0.997 & 1913 & 0.515 & 0.000 & 0 & \underline{0.515} & 0 & 1.375 & 0.000 & 0 & 0 & 1.375 \\
			Average &  &  &  & 98.409 & 1.526 & 4740 & 431.352 & 64 & 24.714 & 1.338 & 16 & 184 & 29.494 \\
			\hline
		\end{tabular}
	}
\end{table} 

\begin{table}[t]
	\centering
	\caption{Overall computational results for comparing the \textit{Path-Based Model} from \cite{ref-dellamico} and \textit{Set-Based Model} for the PCMCA for COMPILERS instances}
	\label{table-comparison-compilers}
	\resizebox{1\linewidth}{0.45\textwidth}{
		\begin{tabular}{lrrrrrrrrrrrrr}
			\hline
			\multicolumn{4}{c}{\multirow{2}{*}{Instance}}          & \multicolumn{4}{c}{Path-Based Model \cite{ref-dellamico}}                        & \multicolumn{6}{c}{Set-Based Model}                                 \\
			\cmidrule(r){5-8} \cmidrule(r){9-14}
			\multicolumn{4}{c}{}                                   & \multicolumn{2}{c}{LR}             & \multicolumn{2}{c}{IP} & \multicolumn{3}{c}{LR}                    & \multicolumn{3}{c}{IP} \\
			\cmidrule(r){1-4} \cmidrule(r){5-6} \cmidrule(r){7-8} \cmidrule(r){9-11} \cmidrule(r){12-14}
			\multicolumn{1}{l}{Name} & Size & Density of P & $z^*$ & \multicolumn{1}{l}{Time (s)} & Gap & Nodes    & Time (s)    & \multicolumn{1}{l}{Cuts} & Time (s) & Gap & Nodes & Cuts & Time (s) \\
			\hline
			gsm.153.124 & 126 & 0.970 & 185 & 0.578 & 0.000 & 0 & 0.578 & 49 & 0.125 & 1.081 & 3 & 53 & 0.140 \\
			gsm.444.350 & 353 & 0.990 & 1542 & 0.078 & 0.000 & 0 & \underline{0.078} & 0 & 0.094 & 0.000 & 0 & 0 & 0.094 \\
			gsm.462.77 & 79 & 0.840 & 292 & 3.422 & 0.000 & 17 & 4.047 & 14 & 0.031 & 0.000 & 0 & 14 & 0.031 \\
			jpeg.1483.25 & 27 & 0.484 & 71 & 0.234 & 0.000 & 43 & 0.266 & 21 & 0.031 & 0.000 & 4 & 34 & 0.047 \\
			jpeg.3184.107 & 109 & 0.887 & 411 & 14.640 & 0.487 & 24 & 16.844 & 32 & 0.093 & 0.000 & 0 & 32 & 0.093 \\
			jpeg.3195.85 & 87 & 0.740 & 13 & 278.844 & 38.462 & 4041 & 1366.985 & 45 & 0.125 & 38.462 & 5674 & 16979 & 897.312 \\
			jpeg.3198.93 & 95 & 0.752 & 140 & 252.734 & 2.857 & 2204 & 529.781 & 29 & 0.141 & 3.571 & 401 & 1686 & 9.704 \\
			jpeg.3203.135 & 137 & 0.897 & 507 & 47.578 & 0.394 & 31 & 56.703 & 18 & 0.094 & 2.170 & 7 & 41 & 0.125 \\
			jpeg.3740.15 & 17 & 0.257 & 33 & 1.782 & 3.030 & 231 & 0.234 & 17 & 0.031 & 0.000 & 0 & 17 & 0.031 \\
			jpeg.4154.36 & 38 & 0.633 & 74 & 0.641 & 5.405 & 1462 & 2.500 & 43 & 0.063 & 0.000 & 0 & 43 & 0.063 \\
			jpeg.4753.54 & 56 & 0.769 & 146 & 2.766 & 0.685 & 11 & 2.984 & 38 & 0.062 & 0.685 & 6 & 59 & 0.109 \\
			susan.248.197 & 199 & 0.939 & 588 & 76.329 & 0.340 & 22 & 106.672 & 21 & 0.125 & 0.000 & 0 & 21 & 0.125 \\
			susan.260.158 & 160 & 0.916 & 472 & 12.156 & 1.695 & 570 & 123.594 & 33 & 0.141 & 0.000 & 0 & 33 & 0.141 \\
			susan.343.182 & 184 & 0.936 & 468 & 194.188 & 1.068 & 776 & 474.391 & 47 & 0.203 & 0.962 & 19 & 89 & 0.359 \\
			typeset.10192.123 & 125 & 0.744 & 241 & 4.859 & 10.373 & 5565 & 297.859 & 93 & 0.500 & 0.000 & 0 & 93 & 0.500 \\
			typeset.10835.26 & 28 & 0.349 & 60 & 0.063 & 0.000 & 0 & 0.063 & 14 & 0.031 & 0.000 & 0 & 14 & 0.031 \\
			typeset.12395.43 & 45 & 0.518 & 125 & 0.531 & 0.800 & 10 & 0.437 & 27 & 0.078 & 0.000 & 0 & 27 & 0.078 \\
			typeset.15087.23 & 25 & 0.557 & 89 & 0.297 & 1.124 & 32 & 0.297 & 24 & 0.047 & 0.000 & 0 & 24 & 0.047 \\
			typeset.15577.36 & 38 & 0.555 & 93 & 0.031 & 0.000 & 0 & 0.031 & 4 & 0.015 & 0.000 & 0 & 4 & 0.015 \\
			typeset.16000.68 & 70 & 0.658 & 67 & 21.891 & 0.000 & 0 & 21.891 & 643 & 3.281 & 8.955 & 144 & 1316 & 7.172 \\
			typeset.1723.25 & 27 & 0.245 & 54 & 0.203 & 5.556 & 7660 & 4.094 & 19 & 0.031 & 5.556 & 21 & 99 & 0.110 \\
			typeset.19972.246 & 248 & 0.993 & 979 & 0.110 & 0.000 & 0 & 0.110 & 0 & 0.062 & 0.000 & 0 & 0 & 0.062 \\
			typeset.4391.240 & 242 & 0.981 & 837 & 378.172 & 0.119 & 46 & 6.250 & 18 & 0.094 & 0.000 & 0 & 18 & 0.094 \\
			typeset.4597.45 & 47 & 0.493 & 133 & 0.437 & 0.000 & 0 & 0.437 & 7 & 0.031 & 0.000 & 0 & 7 & 0.031 \\
			typeset.4724.433 & 435 & 0.995 & 1819 & 4.000 & 0.000 & 0 & 4.000 & 8 & 0.172 & 0.000 & 0 & 8 & 0.172 \\
			typeset.5797.33 & 35 & 0.748 & 93 & 0.234 & 0.000 & 0 & 0.234 & 9 & 0.032 & 0.000 & 0 & 9 & 0.032 \\
			typeset.5881.246 & 248 & 0.986 & 979 & 191.813 & 0.306 & 191 & 356.218 & 52 & 0.343 & 0.000 & 0 & 52 & 0.343 \\
			Average &  &  &  & 55.134 & 2.693 & 849 & 125.095 & 49 & 0.225 & 2.276 & 233 & 769 & 33.965 \\
			\hline
		\end{tabular}
	}
\end{table} 
\end{landscape}

\subsubsection{Computational Results for LR Models}

Tables \ref{table-pcmrt-LR-tsplib}-\ref{table-pcmrt-LR-compilers} show the overall results for the linear relaxation of the MILP models proposed for the PCMCA-WT. In all the tables the \textit{Cost} column reports the value of the objective function. {The \textit{Gap} column indicates the percentage relative difference between the cost of the best known integer solution of the instance ($Cost_{Best}$), and the objective function cost of the model's linear relaxation ($Cost_{LR}$), computed as $100 \cdot \frac{Cost_{Best} - Cost_{LR}}{Cost_{Best}}$.} The \textit{Cuts} column indicates the number of inequalities that are dynamically added to the model, that is {inequalities} (\ref{eqn-18}) and (\ref{eqn-35}) for each model. The solution information are not reported for instances where the solver times out or runs out of memory before finding the optimal solution.

The linear relaxation of the \textit{MCF} model has an average optimality gap of {20.22\%}, and the solver times out before finding the optimal solution for the model's linear relaxation for instances that are {larger} than 240. Comparing the results for the \textit{DA} and \textit{AAC} models, the first model's linear relaxation has an average optimality gap of {23.96\%}, whereas the second model has an average optimaility gap of {23.99\%} across all the instances. Comparing the number of generated cuts, the \textit{AAC} model generates 6\% less cuts compared to the \textit{DA} model. We can notice that the \textit{DA} model finds higher estimates for the optimal integer solution compared to the other two models, however {the} \textit{AAC} model finds better estimates on the symmetrical COMPILERS instances which have symmetric costs. Instances where the \textit{AAC} model {and \textit{MCF} model} found tighter estimates are underlined in the tables.

A major problem that we can notice in the \textit{MCF} model is that the solution times are much larger when compared to the other two models. For example, the \textit{MCF} model finds the optimal solution of ESC78 instance within 9 minutes compared to 4 and 6 seconds of computing time by the other two models. The same increased solution time can be noticed in other instances, sometimes reaching almost an hour to solve the linear relaxation compared to few seconds. For the instances that are optimally solved by all three LR models, the solution time is on average {889} seconds for the \textit{MCF} model, 19 seconds for the \textit{DA} model, and 38 seconds for the \textit{AAC} model. 

In general, it is hard to decide which {linear relaxation} would perform better on some instances, however the \textit{DA} model seems to be the most suitable, as its linear relaxation is much easier to solve compared to the other two, and its result exhibits a lower average optimality gap compared to the other two models.

\begin{landscape}
\begin{table}[t]
	\caption{LR Models computational results for PCMCA-WT for TSPLIB instances}
	\label{table-pcmrt-LR-tsplib}
	\resizebox{1\linewidth}{0.45\textwidth}{
		\begin{tabular}{lrrrrrrrrrrrrr}
			\hline
			\multicolumn{3}{c}{Instance} & 
			\multicolumn{3}{c}{MCF} & \multicolumn{4}{c}{DA}        & \multicolumn{4}{c}{AAC} \\
			\cmidrule(r){1-3} \cmidrule(r){4-6} \cmidrule(r){7-10} \cmidrule(r){11-14} 
			Name   & Size  & Density of P  &  Cost & Time (s) & Gap & Cost & Cuts & Time (s) & Gap & Cost & Cuts & Time (s) & Gap \\ 
			\hline
			br17.10 & 18 & 0.314 & 25.08 & 1.437 & 42.996 & 25.17 & 15 & 0.265 & 42.795 & 25.15 & 18 & 0.203 & 42.841 \\
			br17.12 & 18 & 0.359 & 25.12 & 1.032 & 42.917 & 25.17 & 15 & 0.265 & 42.795 & 25.15 & 18 & 0.203 & 42.841 \\
			ESC07 & 9 & 0.611 & 1887.50 & 0.204 & 0.971 & 1890.75 & 3 & 0.110 & 0.800 & 1782.07 & 7 & 0.031 & 6.502 \\
			ESC11 & 13 & 0.359 & \underline{2127.00} & 0.297 & 2.162 & 2067.00 & 10 & 0.187 & 4.922 & 2040.30 & 8 & 0.312 & 6.150 \\
			ESC12 & 14 & 0.396 & 1138.00 & 0.109 & 0.000 & 1138.00 & 0 & 0.063 & 0.000 & 1138.00 & 1 & 0.078 & 0.000 \\
			ESC25 & 27 & 0.177 & 1043.05 & 3.297 & 9.927 & 1082.41 & 37 & 0.672 & 6.528 & 1064.20 & 40 & 0.890 & 8.100 \\
			ESC47 & 49 & 0.108 & 703.14 & 36.969 & 5.872 & 703.12 & 257 & 9.250 & 5.874 & 703.14 & 80 & 3.625 & 5.871 \\
			ESC63 & 65 & 0.173 & 56.00 & 266.610 & 0.000 & 56.00 & 6 & 1.594 & 0.000 & 56.00 & 67 & 20.937 & 0.000 \\
			ESC78 & 80 & 0.139 & 502.16 & 523.810 & 58.014 & 721.93 & 8 & 4.453 & 39.638 & 718.00 & 6 & 5.969 & 39.967 \\
			ft53.1 & 54 & 0.082 & 3953.05 & 188.391 & 3.325 & 3962.45 & 34 & 5.594 & 3.095 & 3949.66 & 25 & 8.297 & 3.408 \\
			ft53.2 & 54 & 0.094 & 3997.50 & 180.250 & 6.688 & 3998.74 & 40 & 5.531 & 6.659 & 3993.84 & 52 & 8.547 & 6.773 \\
			ft53.3 & 54 & 0.225 & 4286.90 & 171.203 & 21.442 & 4388.35 & 69 & 7.640 & 19.583 & 4249.72 & 97 & 13.562 & 22.124 \\
			ft53.4 & 54 & 0.604 & 5026.27 & 52.062 & 21.940 & 5149.40 & 18 & 4.875 & 20.028 & 5010.26 & 21 & 5.250 & 22.189 \\
			ft70.1 & 71 & 0.036 & 32801.04 & 1021.590 & 1.492 & 32980.40 & 148 & 16.610 & 0.954 & 32851.51 & 130 & 39.453 & 1.341 \\
			ft70.2 & 71 & 0.075 & 32895.06 & 1523.523 & 4.514 & 33016.60 & 160 & 22.235 & 4.161 & 32939.71 & 171 & 48.172 & 4.384 \\
			ft70.3 & 71 & 0.142 & 33441.93 & 2048.220 & 21.740 & 33641.84 & 402 & 47.500 & 21.272 & \underline{33672.54} & 264 & 63.344 & 21.201 \\
			ft70.4 & 71 & 0.589 & 35433.67 & 113.969 & 12.302 & 35805.55 & 132 & 18.188 & 11.381 & 35427.98 & 156 & 31.813 & 12.316 \\
			rbg048a & 50 & 0.444 & \underline{231.57} & 335.875 & 11.277 & 228.06 & 11 & 1.703 & 12.621 & 221.84 & 11 & 3.985 & 15.004 \\
			rbg050c & 52 & 0.459 & 215.12 & 124.781 & 4.393 & 214.35 & 36 & 2.485 & 4.733 & \underline{217.24} & 26 & 3.422 & 3.449 \\
			rbg109 & 111 & 0.909 & 293.13 & 590.328 & 29.196 & 314.83 & 19 & 8.609 & 23.954 & 314.79 & 5 & 13.531 & 23.964 \\
			rbg150a & 152 & 0.927 & 373.34 & 1417.090 & 30.991 & 417.14 & 12 & 10.969 & 22.895 & 416.17 & 7 & 32.969 & 23.074 \\
			rbg174a & 176 & 0.929 & 365.40 & 2096.480 & 37.000 & 405.03 & 10 & 21.984 & 30.167 & 401.07 & 9 & 65.828 & 30.850 \\
			rbg253a & 255 & 0.948 & - & - & - & 458.28 & 11 & 60.750 & 40.714 & 467.20 & 7 & 248.812 & 39.560 \\
			rbg323a & 325 & 0.928 & - & - & - & 920.95 & 23 & 210.250 & 77.176 & 892.63 & 19 & 719.109 & 77.878 \\
			rbg341a & 343 & 0.937 & - & - & - & 677.73 & 52 & 365.250 & 82.165 & 672.90 & 44 & 725.343 & 82.292 \\
			rbg358a & 360 & 0.886 & - & - & - & 699.25 & 77 & 429.547 & 78.785 & 666.92 & 29 & 1395.735 & 79.766 \\
			rbg378a & 380 & 0.894 & - & - & - & 644.63 & 107 & 422.203 & 76.635 & 605.73 & 61 & 1787.078 & 78.045 \\
			kro124p.1 & 101 & 0.046 & 32597.08 & 3482.940 & 7.476 & 32657.90 & 106 & 47.765 & 7.304 & 32603.69 & 123 & 89.266 & 7.457 \\
			kro124p.2 & 101 & 0.053 & 32761.06 & 3482.630 & 13.687 & 33053.63 & 135 & 48.688 & 12.916 & 32922.44 & 171 & 134.109 & 13.262 \\
			kro124p.3 & 101 & 0.092 & 33715.31 & 3490.750 & 37.550 & 33951.74 & 270 & 76.703 & 37.112 & 33826.66 & 303 & 212.000 & 37.344 \\
			kro124p.4 & 101 & 0.496 & 37386.23 & 2552.250 & 32.255 & 38025.91 & 132 & 35.250 & 31.096 & 37233.59 & 174 & 88.859 & 32.532 \\
			p43.1 & 44 & 0.101 & 2825.00 & 864.844 & 36.801 & 2825.00 & 49 & 2.140 & 36.801 & 2797.37 & 53 & 3.797 & 37.419 \\
			p43.2 & 44 & 0.126 & 2759.38 & 1036.547 & 35.453 & 2825.00 & 98 & 2.672 & 33.918 & 2722.91 & 140 & 9.422 & 36.306 \\
			p43.3 & 44 & 0.191 & 2759.53 & 573.968 & 48.660 & 2845.00 & 113 & 3.469 & 47.070 & 2722.79 & 197 & 10.407 & 49.343 \\
			p43.4 & 44 & 0.164 & 2925.07 & 11.937 & 40.305 & 2930.08 & 115 & 2.984 & 40.202 & 2822.27 & 93 & 4.968 & 42.403 \\
			prob.100 & 100 & 0.048 & 643.00 & 3484.390 & 36.210 & 668.13 & 1225 & 598.594 & 33.717 & 657.65 & 1009 & 999.453 & 34.757 \\
			prob.42 & 42 & 0.116 & 148.90 & 57.672 & 12.927 & 153.18 & 107 & 4.813 & 10.421 & 148.26 & 52 & 5.469 & 13.298 \\
			ry48p.1 & 49 & 0.091 & \underline{13134.08} & 99.141 & 4.285 & 13133.93 & 62 & 4.953 & 4.286 & 13115.36 & 54 & 8.391 & 4.421 \\
			ry48p.2 & 49 & 0.103 & 13195.09 & 95.937 & 9.986 & 13243.77 & 48 & 4.703 & 9.654 & 13206.48 & 34 & 5.203 & 9.909 \\
			ry48p.3 & 49 & 0.193 & 13926.14 & 136.859 & 14.700 & 13979.71 & 207 & 12.469 & 14.371 & 13925.41 & 191 & 19.016 & 14.704 \\
			ry48p.4 & 49 & 0.588 & 16168.48 & 16.781 & 17.713 & 16316.13 & 60 & 5.344 & 16.962 & 16186.84 & 93 & 9.406 & 17.620 \\
			Average &  &  &  & 835.671 & 19.921 &  & 108 & 61.691 & 24.784 &  & 99 & 166.982 & 25.626 \\
			\hline
		\end{tabular}
	}
\end{table} 

\begin{table}[t]
	\caption{LR Models computational results for PCMCA-WT for COMPILERS instances}
	\label{table-pcmrt-LR-compilers}
	\resizebox{1\linewidth}{0.45\textwidth}{
		\begin{tabular}{lrrrrrrrrrrrrr}
			\hline
			\multicolumn{3}{c}{Instance} & 
			\multicolumn{3}{c}{MCF} & \multicolumn{4}{c}{DA}        & \multicolumn{4}{c}{AAC} \\ 
			\cmidrule(r){1-3} \cmidrule(r){4-6} \cmidrule(r){7-10} \cmidrule(r){11-14} 
			Name   & Size  & Density of P  &  Cost & Time (s) & Gap & Cost & Cuts & Time (s) & Gap & Cost & Cuts & Time (s) & Gap \\ 
			\hline
			gsm.153.124 & 126 & 0.97 & 221.14 & 135.500 & 29.348 & 222.23 & 15 & 0.610 & 29.000 & \underline{223.41} & 15 & 3.312 & 28.623 \\
			gsm.444.350 & 353 & 0.99 & - & - & - & 1914.83 & 6 & 4.531 & 33.351 & 2042.75 & 4 & 5.156 & 28.898 \\
			gsm.462.77 & 79 & 0.84 & 377.54 & 231.625 & 22.636 & 384.41 & 27 & 6.016 & 21.227 & 380.96 & 29 & 5.375 & 21.934 \\
			jpeg.1483.25 & 27 & 0.484 & \underline{84.00} & 1.844 & 3.446 & 78.97 & 17 & 0.406 & 9.230 & 76.89 & 16 & 0.719 & 11.621 \\
			jpeg.3184.107 & 109 & 0.887 & 419.22 & 254.187 & 38.710 & 441.65 & 60 & 2.875 & 35.431 & \underline{451.07} & 76 & 16.047 & 34.054 \\
			jpeg.3195.85 & 87 & 0.74 & \underline{13.04} & 3595.590 & 47.837 & 9.00 & 126 & 7.875 & 64.000 & 13.00 & 195 & 9.130 & 48.000 \\
			jpeg.3198.93 & 95 & 0.752 & 140.26 & 3594.700 & 31.244 & 151.87 & 214 & 9.296 & 25.554 & \underline{152.79} & 152 & 11.730 & 25.103 \\
			jpeg.3203.135 & 137 & 0.897 & 524.22 & 1217.125 & 30.104 & 564.03 & 58 & 3.234 & 24.796 & \underline{568.97} & 122 & 21.063 & 24.137 \\
			jpeg.3740.15 & 17 & 0.257 & 33.00 & 0.313 & 0.000 & 33.00 & 5 & 0.093 & 0.000 & 33.00 & 3 & 0.125 & 0.000 \\
			jpeg.4154.36 & 38 & 0.633 & \underline{86.88} & 7.843 & 3.469 & 85.06 & 26 & 2.125 & 5.489 & 84.01 & 21 & 0.765 & 6.656 \\
			jpeg.4753.54 & 56 & 0.769 & 150.20 & 76.875 & 8.413 & 153.08 & 30 & 3.500 & 6.659 & 150.19 & 40 & 2.250 & 8.421 \\
			susan.248.197 & 199 & 0.939 & 613.41 & 3519.028 & 48.192 & 658.84 & 108 & 9.672 & 44.355 & \underline{682.70} & 138 & 34.656 & 42.340 \\
			susan.260.158 & 160 & 0.916 & 494.65 & 1681.391 & 43.533 & 519.01 & 116 & 7.796 & 40.752 & \underline{534.43} & 244 & 55.359 & 38.992 \\
			susan.343.182 & 184 & 0.936 & 469.79 & 1488.110 & 45.500 & 539.47 & 72 & 6.156 & 37.416 & \underline{554.04} & 94 & 20.234 & 35.726 \\
			typeset.10192.123 & 125 & 0.744 & 246.52 & 3579.440 & 40.599 & 264.30 & 131 & 22.078 & 36.313 & 260.60 & 90 & 23.328 & 37.205 \\
			typeset.10835.26 & 28 & 0.349 & \underline{93.55} & 2.187 & 16.470 & 81.83 & 7 & 0.328 & 26.938 & 92.34 & 9 & 0.625 & 17.554 \\
			typeset.12395.43 & 45 & 0.518 & \underline{139.02} & 57.437 & 4.784 & 137.85 & 110 & 3.094 & 5.582 & 137.27 & 107 & 4.484 & 5.979 \\
			typeset.15087.23 & 25 & 0.557 & 92.27 & 2.046 & 4.878 & 93.00 & 13 & 0.157 & 4.124 & 93.00 & 30 & 0.516 & 4.124 \\
			typeset.15577.36 & 38 & 0.555 & 120.01 & 4.141 & 3.995 & 120.69 & 21 & 0.531 & 3.448 & 120.01 & 14 & 1.015 & 3.992 \\
			typeset.16000.68 & 70 & 0.658 & 70.00 & 2051.330 & 12.500 & 70.07 & 121 & 4.062 & 12.413 & 69.48 & 486 & 32.735 & 13.150 \\
			typeset.1723.25 & 27 & 0.245 & \underline{56.00} & 5.516 & 6.667 & 55.33 & 93 & 1.062 & 7.783 & 55.50 & 90 & 2.047 & 7.500 \\
			typeset.19972.246 & 248 & 0.993 & - & - & - & 1229.52 & 4 & 1.891 & 36.261 & 1234.43 & 11 & 4.328 & 36.007 \\
			typeset.4391.240 & 242 & 0.981 & - & - & - & 1006.12 & 31 & 2.812 & 28.745 & 1057.66 & 64 & 10.281 & 25.095 \\
			typeset.4597.45 & 47 & 0.493 & \underline{144.01} & 23.484 & 7.094 & 143.13 & 89 & 3.282 & 7.658 & 141.18 & 169 & 9.281 & 8.916 \\
			typeset.4724.433 & 435 & 0.995 & - & - & - & 2351.03 & 29 & 12.016 & 31.517 & 2351.76 & 33 & 21.531 & 31.495 \\
			typeset.5797.33 & 35 & 0.748 & 105.93 & 2.843 & 6.258 & 106.00 & 9 & 0.625 & 6.195 & 104.21 & 21 & 0.891 & 7.779 \\
			typeset.5881.246 & 248 & 0.986 & - & - & - & 1204.29 & 27 & 5.234 & 29.159 & 1229.51 & 28 & 9.422 & 27.676 \\
			Average &  &  &  & 978.753 & 20.713 &  & 58 & 4.495 & 22.718 &  & 85 & 11.348 & 21.518 \\
			\hline
		\end{tabular}
	}
\end{table} 
\end{landscape}

\subsubsection{Computational Results for IP Models}

Tables \ref{table-pcmrt-IP-tsplib}-\ref{table-pcmrt-IP-compilers} show the overall results of the MILP models proposed for the PCMCA-WT. {In Tables \ref{table-pcmrt-IP-tsplib}-\ref{table-pcmrt-IP-compilers}, the \textit{Cost} column reports the lower bound ($LB$) and upper bound ($UB$) on the value of the objective function obtained from solving the respective model. The \textit{IP Gap} column measures the percentage relative difference between the upper and lower bound obtained from solving the respective model, calculated as $100 \cdot \frac{UB - LB}{UB}$.} \textit{Cuts} column indicates the number {of} inequalities that are dynamically added to the model, that is inequalities (\ref{eqn-18}) and (\ref{eqn-35}) for each model. The solution time is not reported in the tables for the instances that are not optimally solved within the time limit. Moreover, the solution information is not reported for instances where it was not possible to solve the model's linear relaxation (see Tables \ref{table-pcmrt-LR-tsplib}-\ref{table-pcmrt-LR-compilers}).

The \textit{MCF} model achieves an optimality gap of {29\%} on average across all the instances it is able to solve, compared to 14\% for the \textit{DA} model, and 15\% for the \textit{AAC} model. We can notice that the solver is finding it difficult to solve the linear relaxation of the MCF model by observing the small number of nodes generated in the time limit for some of the instances, compared to the other two models. However, the \textit{MCF} model is able to provide the best known solution (underlined in the tables) for {1} out of 68 instances, and the best lower bound for 1 instance.

Comparing the results of the \textit{DA} and \textit{AAC} models, we can see that the \textit{DA} model has an average optimality gap of 19\% compared to 21\% across all the instances. Furthermore, the \textit{DA} model solves {three} extra instances for the TSPLIB instances, while the \textit{AAC} model solves one extra instance for the COMPILERS instances. For the instances {(marked bold in the tables)} that are optimally solved by both models, the \textit{DA} model is {119}\% faster at solving those instances on average, {computed as $100 \cdot \frac{T_{DA}-T_{AAC}}{T_{DA}}$, where $T_{DA}$ and $T_{AAC}$ is the average solution time for solving those instances using the $DA$ model and $AAC$ model.} Comparing the number of branch-decision-tree nodes generated by the two models, the \textit{DA} model generates 36\% more nodes on average across all the instances, and 10\% more nodes on average for the instances that are optimally solved by the two models, {The \textit{DA} model found the best lower bound for {55} out of 68 instances, and better upper bounds for {51} out of 68 instances. On the other hand,} the \textit{AAC} model found tighter lower bounds for {12} out of 68 instances, and better upper bounds for {16} of those instances (underlined in the tables). {In general, the \textit{DA} model performs better than the \textit{AAC} model, except on symmetrical instances and/or instances with extreme high densities larger than 0.9.}

We can conclude that the \textit{AAC} model is more suitable for symmetrical instances with extreme densities, while the \textit{DA} model is more suitable for general instances. The \textit{MCF} finds better bounds compared to the other two models for some instances, however it is not suitable for instances with size larger than 200.  

\begin{landscape}
	\begin{table}[t]
		\caption{IP Models computational results for PCMCA-WT for TSPLIB instances}
		\label{table-pcmrt-IP-tsplib}
		\resizebox{\linewidth}{!}{
			\begin{tabular}{lrrrrrrrrrrrrrrrr}
				\hline
				\multicolumn{3}{c}{Instance} & \multicolumn{4}{c}{MCF} & \multicolumn{5}{c}{DA}    & \multicolumn{5}{c}{AAC}  \\
				\cmidrule(r){1-3} \cmidrule(r){4-7} \cmidrule(r){8-12} \cmidrule(r){13-17} 
				Name   & Size  & Density of P  & Cost & IP Gap & Nodes & Time (s) & Cost & IP Gap & Nodes & Cuts & Time (s) & Cost & IP Gap & Nodes & Cuts & Time (s)  \\
				\hline
				br17.10  & 18 & 0.314 & [\underline{35}, 44] & 20.455 & 192505 &  &  [34, 44]  & 22.727 & 2767833 & 3799 &    &  [34, 44]  & 22.727 & 2074173 & 3893 &  \\
				br17.12  & 18 & 0.359 & [35, 44] & 20.455 & 291653 &  &  [35, 45]  & 22.222 & 3376744 & 2519 &    &  [35, 44]  & 20.455 & 2845672 & 2265 &  \\  
				\textbf{ESC07}  & 9 & 0.611 & 1906 & 0.000 & 7 & 0.109 & 1906 & 0.000 & 17 & 5 & 0.078 & 1906 & 0.000 & 0 & 3 & 0.078 \\
				\textbf{ESC11}  & 13 & 0.359 & 2174 & 0.000 & 20 & 0.265 & 2174 & 0.000 & 22 & 13 & 0.172 & 2174 & 0.000 & 33 & 12 & 0.219 \\
				\textbf{ESC12}  & 14 & 0.396 & 1138 & 0.000 & 0 & 0.078 & 1138 & 0.000 & 0 & 0 & 0.063 & 1138 & 0.000 & 0 & 1 & 0.078 \\
				\textbf{ESC25}  & 27 & 0.177 & 1158 & 0.000 & 1158 & 26.625 & 1158 & 0.000 & 1450 & 698 & 3.922 & 1158 & 0.000 & 1098 & 452 & 4.812 \\
				ESC47  & 49 & 0.108 & 747 & 0.000 & 2776 & 3154.14 & 747 & 0.000 & 750 & 31743 & 2230.171 &  [704, 783]  & 10.089 & 2 & 37495 &     \\
				\textbf{ESC63}  & 65 & 0.173 & [56, 61] & 8.197 & 165 &  & 56 & 0.000 & 4100 & 287 & 53.281 & 56 & 0.000 & 3177 & 14118 & 1223.313 \\
				\textbf{ESC78}  & 80 & 0.139 & [922, 1346] & 31.501 & 920 &  & 1196 & 0.000 & 11177 & 267 & 159.093 & 1196 & 0.000 & 15734 & 943 & 301.625 \\
				\textbf{ft53.1}  & 54 & 0.082 & [3984, 4089] & 2.568 & 790 &  & 4089 & 0.000 & 134416 & 2932 & 1430.797 & 4089 & 0.000 & 56905 & 1433 & 1195.875 \\
				ft53.2  & 54 & 0.094 & [4054, 4659] & 12.986 & 1055 &  &  [4135, 4284]  & 3.478 & 200163 & 6053 &    &  [4112, 4318]  & 4.771 & 57840 & 8393 &     \\
				ft53.3  & 54 & 0.225 & [4472, 6071] & 26.338 & 1030 &  &  [4623, 5457]  & 15.283 & 168357 & 6749 &    &  [4545, 5734]  & 20.736 & 92513 & 5778 &     \\
				ft53.4  & 54 & 0.604 & [5409, 6871] & 21.278 & 6531 &  &  [5657, 6439]  & 12.145 & 564462 & 2361 &    &  [5559, 6668]  & 16.632 & 340573 & 1588 &     \\
				ft70.1  & 71 & 0.036 & [32937, 37078] & 11.168 & 807 &  &  [33117, 33298]  & 0.544 & 23240 & 11362 &    &  [\underline{33128}, 33714]  & 1.738 & 27834 & 6322 &     \\
				ft70.2  & 71 & 0.075 & [33084, 43282] & 23.562 & 408 &  &  [33357, 34450]  & 3.173 & 11000 & 13238 &    &  [33297, 34713]  & 4.079 & 11474 & 10171 &     \\
				ft70.3  & 71 & 0.142 & [33693,  66153] & 49.068 & 287 &  &  [33914, 42732]  & 20.636 & 8574 & 15867 &    &  [33830, 49116]  & 31.122 & 9040 & 11611 &     \\
				ft70.4  & 71 & 0.589 & [35912, 44507] & 19.312 & 1969 &  &  [36517, 40404]  & 9.620 & 121796 & 3615 &    &  [36188, 40701]  & 11.088 & 78814 & 3008 &     \\
				rbg048a  & 50 & 0.444 & [259, 267] & 2.996 & 596 &  & 261 & 0.000 & 146017 & 15025 & 2794.06 &  [260, 266]  & 2.256 & 62610 & 14794 &     \\
				\textbf{rbg050c}  & 52 & 0.459 & [225, 236] & 4.661 & 847 &  & 225 & 0.000 & 1260 & 326 & 10.469 & 225 & 0.000 & 24050 & 432 & 210.562 \\
				rbg109  & 111 & 0.909 & [320, 699] & 54.220 & 529 &  &  [354, 414]  & 14.493 & 150506 & 219 &    &  [349, 428]  & 18.458 & 85685 & 601 &     \\
				rbg150a  & 152 & 0.927 & [403, 972] & 58.539 & 384 &  &  [447, 541]  & 17.375 & 59431 & 232 &    &  [441, 650]  & 32.154 & 34702 & 618 &     \\
				rbg174a  & 176 & 0.929 & [400, 1121] & 64.318 & 326 &  &  [446, 580]  & 23.103 & 42286 & 302 &    &  [441, 627]  & 29.665 & 24226 & 113 &     \\
				rbg253a  & 255 & 0.948 & - & - & - &  &  [477, 773]  & 38.292 & 12700 & 152 &    &  [475, 971]  & 51.081 & 5427 & 231 &     \\
				rbg323a  & 325 & 0.928 & - & - & - &  &  [926, 4367]  & 78.796 & 4536 & 86 &    &  [903, \underline{4035}]  & 77.621 & 2845 & 78 &     \\
				rbg341a  & 343 & 0.937 & - & - & - &  &  [681, 3832]  & 82.229 & 3963 & 116 &    &  [680, \underline{3800}]  & 82.105 & 2300 & 90 &     \\
				rbg358a  & 360 & 0.886 & - & - & - &  &  [706, 3296]  & 78.580 & 3408 & 698 &    &  [674, 4109]  & 83.597 & 1779 & 144 &     \\
				rbg378a  & 380 & 0.894 & - & - & - &  &  [649, 2759]  & 76.477 & 2831 & 110 &    &  [613, 6701]  & 90.852 & 1128 & 106 &     \\
				kro124p.1  & 101 & 0.046 & [32598, 45652] & 28.595 & 1 &  &  [32858, 35438]  & 7.280 & 3389 & 139 &    &  [32827, \underline{35231}]  & 6.824 & 6926 & 5724 &     \\
				kro124p.2  & 101 & 0.053 & [32762, 55265] & 40.718 & 1 &  &  [33190, 37956]  & 12.557 & 14237 & 5920 &    &  [33163, 38002]  & 12.734 & 8410 & 3458 &    \\ 
				kro124p.3  & 101 & 0.092 & [33716, 329579] & 89.770 & 1 &  &  [34217, 53988]  & 36.621 & 20016 & 6098 &    &  [33991, 77266]  & 56.008 & 3313 & 9747 &     \\
				kro124p.4  & 101 & 0.496 & [37395, 103491] & 63.866 & 2 &  &  [39413, 55187]  & 28.583 & 71246 & 2652 &    &  [38269, 62162]  & 38.437 & 32120 & 2746 &     \\
				p43.1  & 44 & 0.101 & [2825, \underline{4470}] & 36.801 & 779 &  &  [2827, 4585]  & 38.342 & 160185 & 19944 &    &  [2825, 4615]  & 38.787 & 51091 & 11280 &     \\
				p43.2  & 44 & 0.126 & [2760, 4845] & 43.034 & 430 &  &  [2826, 4275]  & 33.895 & 162580 & 18401 &    &  [2826, 5025]  & 43.761 & 55765 & 15014 &     \\
				p43.3  & 44 & 0.191 & [2828, 6955] & 59.339 & 511 &  &  [2864, 6105]  & 53.088 & 85658 & 16311 &    &  [2847, \underline{5375}]  & 47.033 & 26582 & 10567 &     \\
				p43.4  & 44 & 0.164 & [2991, 5415] & 44.765 & 5826 &  &  [3101, 4900]  & 36.714 & 875030 & 2702 &    &  [2981, 4970]  & 40.020 & 408540 & 1620 &     \\
				prob.100  & 100 & 0.048 & [643, 3611] & 82.193 & 1 &  &  [674, 1047]  & 35.626 & 3718 & 9961 &    &  [661, \underline{1008}]  & 34.425 & 1533 & 7323 &     \\
				prob.42  & 42 & 0.116 & [163, 171] & 4.678 & 2746 &  & 171 & 0.000 & 112468 & 7561 & 945.203 &  [167, 171]  & 2.339 & 157734 & 11472 &     \\
				ry48p.1  & 49 & 0.091 & [13278, 15362] & 13.566 & 1339 &  &  [13356, 13854]  & 3.595 & 158744 & 8620 &    &  [\underline{13371}, \underline{13722}]  & 2.558 & 85005 & 6664 &     \\
				ry48p.2  & 49 & 0.103 & [13403, 17833] & 24.842 & 585 &  &  [13507, 16968]  & 20.397 & 82990 & 10129 &    &  [\underline{13508}, \underline{14659}]  & 7.852 & 63149 & 6161 &     \\
				ry48p.3  & 49 & 0.193 & [13984, 21379] & 34.590 & 773 &  &  [14371, 16337]  & 12.034 & 70000 & 8508 &    &  [14299, \underline{16326}]  & 12.416 & 38332 & 6103 &     \\
				ry48p.4  & 49 & 0.588 & [16922, 20030] & 15.517 & 10691 &  &  [17339, 19649]  & 11.756 & 699363 & 921 &    &  [17136, 20915]  & 18.068 & 398622 & 992 &     \\
				Average  &    &    &  & 28.164 & 14679 &  &    & 20.723 & 252211 & 5772 & 693.392 &    & 23.719 & 175531 & 5453 & 367.070 \\
				\hline
			\end{tabular}
		}
	\end{table} 
	
	\begin{table}[t]
		\caption{IP Models computational results for PCMCA-WT for COMPILERS instances}
		\label{table-pcmrt-IP-compilers}
		\resizebox{\linewidth}{!}{
			\begin{tabular}{lrrrrrrrrrrrrrrrr}
				\hline
				\multicolumn{3}{c}{Instance} & \multicolumn{4}{c}{MCF} & \multicolumn{5}{c}{DA}    & \multicolumn{5}{c}{AAC}  \\
				\cmidrule(r){1-3} \cmidrule(r){4-7} \cmidrule(r){8-12} \cmidrule(r){13-17}
				Name   & Size  & Density of P  & Cost & IP Gap & Nodes & Time (s) & Cost & IP Gap & Nodes & Cuts & Time (s) & Cost & IP Gap & Nodes & Cuts & Time (s)  \\
				\hline
				gsm.153.124  & 126 & 0.97 & [234, 331] & 29.305 & 936 &  &  [246, 313]  & 21.406 & 357665 & 413 &    &  [243, 319]  & 23.824 & 270962 & 405 &  \\
				gsm.444.350  & 353 & 0.99 & - & - & - &  &  [1996, 2873]  & 30.526 & 50228 & 209 &    &  [\underline{2103}, 2905]  & 27.608 & 27606 & 216 &     \\
				gsm.462.77  & 79 & 0.84 & [392, 707] & 44.554 & 1150 &  &  [396, 493]  & 19.675 & 462000 & 1734 &    &  [391, \underline{488}]  & 19.877 & 185939 & 2399 &     \\
				\textbf{jpeg.1483.25}  & 27 & 0.484 & 87 & 0.000 & 2338 & 71.61 & 87 & 0.000 & 20708 & 583 & 11.254 & 87 & 0.000 & 33035 & 553 & 72.907 \\
				jpeg.3184.107  & 109 & 0.887 & [430, 811] & 46.979 & 992 &  &  [488, 684]  & 28.655 & 315065 & 1051 &    &  [\underline{489}, 684]  & 28.509 & 121292 & 744 &     \\
				jpeg.3195.85  & 87 & 0.74 & [14, 76] & 81.579 & 0 &  &  [22, 25]  & 12.000 & 27085 & 9849 &    &  [21, 30]  & 30.000 & 2955 & 3273 &     \\
				jpeg.3198.93  & 95 & 0.752 & [141, 353] & 60.057 & 0 &  &  [172, 213]  & 19.249 & 89431 & 2584 &    &  [161, \underline{204}]  & 21.078 & 16822 & 3448 &     \\
				jpeg.3203.135  & 137 & 0.897 & [535, 1539] & 65.237 & 14 &  &  [600, 755]  & 20.530 & 179785 & 2533 &    &  [595, \underline{750}]  & 20.667 & 120417 & 1332 &  \\   
				\textbf{jpeg.3740.15}  & 17 & 0.257 & 33 & 0.000 & 0 & 0.266 & 33 & 0.000 & 0 & 5 & 0.093 & 33 & 0.000 & 0 & 3 & 0.125 \\
				\textbf{jpeg.4154.36}  & 38 & 0.633 & 90 & 0.000 & 10271 & 139.579 & 90 & 0.000 & 15705 & 461 & 25.766 & 90 & 0.000 & 4242 & 298 & 12.562 \\
				jpeg.4753.54  & 56 & 0.769 & [157, 174] & 9.770 & 1826 &  &  [163, 165]  & 1.212 & 1128020 & 1079 &    & \underline{164} & 0.000 & 594353 & 911 & 2231.235 \\
				susan.248.197  & 199 & 0.939 & [614, 3014] & 79.628 & 0 &  &  [718, 1184]  & 39.358 & 62265 & 1519 &    &  [\underline{736}, 1353]  & 45.602 & 39582 & 1511 & \\    
				susan.260.158  & 160 & 0.916 & [498, 2530] & 80.316 & 74 &  &  [541, 1149]  & 52.916 & 163071 & 3999 &    &  [\underline{564}, 876]  & 35.616 & 49000 & 1510 &   \\  
				susan.343.182  & 184 & 0.936 & [527, 1433] & 63.224 & 5 &  &  [586, 862]  & 32.019 & 111339 & 901 &    &  [\underline{591}, 887]  & 33.371 & 55306 & 1143 &     \\
				typeset.10192.123  & 125 & 0.744 & [247, 774] & 68.088 & 0 &  &  [280, 415]  & 32.530 & 92153 & 3643 &    &  [280, 456]  & 38.596 & 59000 & 1924 &     \\
				typeset.10835.26  & 28 & 0.349 & [99, 114] & 13.158 & 119310 &  &  [99, 112]  & 11.607 & 986681 & 7100 &    &  [99, 113]  & 12.389 & 577573 & 5961 &     \\
				typeset.12395.43  & 45 & 0.518 & [141, 148] & 4.730 & 5556 &  &  [143, 146]  & 2.055 & 392093 & 4782 &    &  [141, 147]  & 4.082 & 175106 & 4776 &     \\
				\textbf{typeset.15087.23}  & 25 & 0.557 & 97 & 0.000 & 10917 &  & 97 & 0.000 & 24721 & 1082 & 29.235 & 97 & 0.000 & 13225 & 759 & 32.094 \\
				\textbf{typeset.15577.36}  & 38 & 0.555 & 125 & 0.000 & 12472 & 1738.422 & 125 & 0.000 & 18600 & 1467 & 51.688 & 125 & 0.000 & 106552 & 3042 & 834.797 \\
				typeset.16000.68  & 70 & 0.658 & [71, 102] & 30.392 & 2 &  &  [77, 86]  & 10.465 & 13917 & 28566 &    &  [77, 80]  & 3.750 & 29319 & 6506 &     \\
				\textbf{typeset.1723.25}  & 27 & 0.245 & 60 & 0.000 & 682 & 84.75 & 60 & 0.000 & 1212 & 281 & 3.391 & 60 & 0.000 & 5533 & 481 & 29.734 \\
				typeset.19972.246  & 248 & 0.993 & - & - & - &  &  [1325, 1963]  & 32.501 & 68601 & 48 &    &  [1307, \underline{1929}]  & 32.245 & 33530 & 45 &     \\
				typeset.4391.240  & 242 & 0.981 & - & - & - &  &  [1067, 1419]  & 24.806 & 100406 & 332 &    &  [\underline{1093}, \underline{1412}]  & 22.592 & 43577 & 388 &     \\
				typeset.4597.45  & 47 & 0.493 & [147, 167] & 11.976 & 3247 &  &  [150, 155]  & 3.226 & 501643 & 4681 &    &  [149, 159]  & 6.289 & 107981 & 3505 &     \\
				typeset.4724.433  & 435 & 0.995 & - & - & - &  &  [2378, 5376]  & 55.766 & 29538 & 214 &    &  [\underline{2460}, \underline{3433}]  & 28.343 & 15907 & 211 &     \\
				\textbf{typeset.5797.33}  & 35 & 0.748 & 113 & 0.000 & 2872 & 100.234 & 113 & 0.000 & 15172 & 652 & 24.25 & 113 & 0.000 & 10733 & 341 & 24.953 \\
				typeset.5881.246  & 248 & 0.986 & - & - & - &  &  [1258, 1877]  & 32.978 & 105414 & 318 &    &  [\underline{1305}, \underline{1700}]  & 23.235 & 44716 & 183 &   \\  
				Average  &    &    &  & 31.318 & 7848 &  &    & 17.907 & 197501 & 2966 & 20.811 &    & 16.951 & 101639 & 1699 & 404.801 \\
				\hline
			\end{tabular}
		}
	\end{table} 
\end{landscape}

\section{Conclusions}
\label{sec-conclusions}

This work introduces a more scalable model for the PCMCA, and three models for solving the PCMCA-WT. A proof of complexity shows that {the two problems fall inside the {{\sc NP}-hard} complexity class}. The experimental results show that the benchmark instances for the PCMCA could be solved much more efficiently compared to a previously proposed model. Moreover, the results show that the benchmark instances are much harder to solve under the PCMCA-WT settings. In general, large sized instances with dense precedence graphs are outside the reach of the three PCMCA-WT models proposed, therefore further studies to find better formulations for the problem are needed.

The two problems proposed have applications in designing distribution networks for commodities such as oil and gas, where the network is designed in such a way {to avoid} passing through a certain location when delivering the commodity to another location. Moreover, the delivery can be scheduled to reach certain locations with higher priority first. The Sequential Ordering Problem is a related problem with a wide range of applications in the domains of scheduling and logistics. {The two problems proposed can be seen as relaxations of the Sequential Ordering Problem. {Therefore, lower bounds for the SOP based on them (or their linear relaxations) could be derived, together with some new valid inequalities.} This in turn could have impact on solving the several real-world problems that can be formulated as a Sequential Ordering Problem.}

\section*{Acknowledgment}
We would like to thank the anonymous reviewers for their constructive comments and suggestions.}





\appendix

\section{{The \emph{Path-Based} model from \cite{ref-dellamico} fails to detect the violation of precedence constraints in fractional solutions}}\label{a1}

The MILP model previously proposed for the PCMCA in \cite{ref-dellamico}, imposes the precedence relationships by propagating a value along every path with end-points $s$ and $t$ for $(s,t)\in R$ in order to detect a precedence violation. The violation is detected if a non-zero value is propagated from $t$ down to $s$. The model sometimes fail to detect a violating path in a fractional solution due to the diminishing propagated values along that path \cite{ref-dellamico}. However, the \textit{Set-Based Model} proposed in Section 2.2 is able to detect the violation for the same fractional solution. In order to show this, consider the example
shown in Figure \ref{fig-p_vs_s}. Referring back to Algorithm 1, we construct the graph $D_j$ (consisting of all the vertices of the original graph $G$ excluding the successors of vertex $s$ in the precedence graph $P$), then we compute a minimum $(r, s)$-cut in $D_j$, namely cut $[\{s, 1, 2\}, \{r\}]$, which has a value less than 1 in the example. On the other hand in the model proposed in \cite{ref-dellamico} a value of 0 will be propagated down from $t$ to $s$ failing to detect the violating path. The details about how the value is propagated can be found in \cite{ref-dellamico}.

\begin{figure}[H]
	\centering
	\hspace{-4mm}%
	\begin{subfigure}{.3\textwidth}
		\centering
		\begin{tikzpicture}[node distance={2cm}, main/.style = {draw, circle},scale=0.9, transform shape]
			\node[main] (0) {$r$}; 
			\node[main] (1) [below of=0] {$t$};
			\node[main] (2) [below right of=1] {1};
			\node[main] (3) [below left of=1] {2};
			\node[main] (4) [below right of=3] {$s$};
			
			\path[->,>=stealth, line width=0.6] (0) edge [auto=left] node {1.0} (1);
			\path[->,>=stealth, line width=0.6] (0) edge [bend right, auto=right] node {0.5} (3);
			\path[->,>=stealth, line width=0.6] (1) edge [auto=left] node {0.5} (2);
			\path[->,>=stealth, line width=0.6] (1) edge [auto=right] node {0.5} (3);
			\path[->,>=stealth, line width=0.6] (3) edge [auto=left] node {0.5} (2);
			\path[->,>=stealth, line width=0.6] (3) edge [auto=right] node {0.5} (4);
			\path[->,>=stealth, line width=0.6] (2) edge [auto=left] node {0.5} (4);
			
			\path[dashed, ->,>=stealth, color=red, line width=0.6] (4) edge [bend right, auto=right] node {} (1);
		\end{tikzpicture}
	\end{subfigure}\hspace{2mm}%
	\begin{subfigure}{.3\textwidth}
		\centering
		\begin{tikzpicture}[node distance={2cm}, main/.style = {draw, circle},scale=0.9, transform shape]
			\node[main] (0) {$r$}; 
			\node[main, label=right:\textcolor{blue}{1}] (1) [below of=0] {$t$};
			\node[main, label=right:\textcolor{blue}{0.5}] (2) [below right of=1] {1};
			\node[main, label=left:\textcolor{blue}{0.5}] (3) [below left of=1] {2};
			\node[main, label=right:\textcolor{blue}{0}] (4) [below right of=3] {$s$};
			
			\path[->,>=stealth, line width=0.6] (0) edge [auto=left] node {1.0} (1);
			\path[->,>=stealth, line width=0.6] (0) edge [bend right, auto=right] node {0.5} (3);
			\path[->,>=stealth, line width=0.6] (1) edge [auto=left] node {0.5} (2);
			\path[->,>=stealth, line width=0.6] (1) edge [auto=right] node {0.5} (3);
			\path[->,>=stealth, line width=0.6] (3) edge [auto=left] node {0.5} (2);
			\path[->,>=stealth, line width=0.6] (3) edge [auto=right] node {0.5} (4);
			\path[->,>=stealth, line width=0.6] (2) edge [auto=left] node {0.5} (4);
			
			\path[dashed, ->,>=stealth, color=red, line width=0.6] (4) edge [bend right, auto=right] node {} (1);
		\end{tikzpicture}
	\end{subfigure}\hspace{7mm}%
	\begin{subfigure}{.3\textwidth}
		\centering
		\begin{tikzpicture}[node distance={2cm}, main/.style = {draw, circle},scale=0.9, transform shape]
			\node[main] (0) {$r$}; 
			\node[main] (2) [below right of=1] {1};
			\node[main] (3) [below left of=1] {2};
			\node[main] (4) [below right of=3] {$s$};
			
			\path[->,>=stealth, line width=0.6] (0) edge [bend right, auto=right] node {0.5} (3);
			\path[->,>=stealth, line width=0.6] (3) edge [auto=left] node {0.5} (2);
			\path[->,>=stealth, line width=0.6] (3) edge [auto=right] node {0.5} (4);
			\path[->,>=stealth, line width=0.6] (2) edge [auto=left] node {0.5} (4);
			
			\draw [dashed,red, line width=0.6] (2,-3.2) to[out=90,in=90] (-2, -3.2);
		\end{tikzpicture}
	\end{subfigure}
	\caption{Example of a fractional solution that violates the precedence relationship $(s,t) \in R$, and how the violation cannot be detected by the \textit{Path-Based Model}. The figure on the left shows the fractional solution with the violating path from $t$ to $s$. The figure in the middle shows the value propagated from $t$ to $s$ next to each vertex, where the violation is not detected using the \textit{Path-Based Model} because a value of 0 is propagated down to vertex $s$ from $t$. The figure on the right shows how the \textit{Set-Based Model} is able to detect the violating path, since using Algorithm 1 the graph has a minimum $(r,s)$-cut in $D_j$ with value less than 1.}
	\label{fig-p_vs_s}
\end{figure}

\bibliographystyle{apa-good}
\bibliography{ref}

\end{document}